\documentclass[12pt]{article}

\newlength{\spse}
\usepackage{amsfonts}
\usepackage{amsthm}
\input{epsf.sty}
\usepackage{latexsym,amsmath,amsfonts,amscd}
\usepackage{epsfig}
\usepackage{changebar}
\usepackage{pstricks}
\usepackage{pst-plot}
\usepackage{multirow}
\usepackage{subfigure}
\usepackage{placeins}
\usepackage{amssymb}
\usepackage{mathrsfs}
\usepackage[title]{appendix}
\usepackage{bm}
\usepackage{hyperref}

\usepackage{multirow,bigstrut}
\usepackage{enumitem}
\usepackage{booktabs}

\numberwithin{equation}{section}

\theoremstyle{plain}
\newtheorem{thm}{Theorem}[section]
\newtheorem{lem}[thm]{Lemma}

\newtheorem{prop}[thm]{Proposition}

\newtheorem{rem}[thm]{Remark}

\theoremstyle{definition}
\newtheorem{exam}[thm]{Example}

\newcommand{\brac}[1]{\left(#1\right)}
\newcommand{\abs}[1]{\left\vert#1\right\vert}
\newcommand{\norm}[1]{\left\Vert#1\right\Vert}
\def\half{\frac 1 2}

\newcommand{\bk}{\mathbf{k}}
\newcommand{\bx}{\mathbf{x}}

\newfont{\iams}{msbm9}

\newcommand{\commentbis}[1]{}
\newcommand{\be}{\begin{eqnarray}}
\newcommand{\ee}{\end{eqnarray}}
\newcommand{\beno}{\begin{eqnarray*}}
\newcommand{\eeno}{\end{eqnarray*}}

\newcommand{\barr}[1]{\begin{array}{#1}}
\newcommand{\earr}{\end{array}}

\newcommand{\beq}{\begin{equation}}
\newcommand{\eeq}{\end{equation}}
\newcommand{\beqa}{\begin{eqnarray}}
\newcommand{\eeqa}{\end{eqnarray}}

\newcommand{\bv}{{\bf v}}

\newcommand{\bV}{{\bf V}}
\newcommand{\bn}{{\bf n}}

\newcommand{\bzero}{\mathbf{0}}
\newcommand{\bone}{\mathbf{1}}
\newcommand{\bl}{\mathbf{l}}
\newcommand{\bi}{\mathbf{i}}

\newcommand{\bj}{\mathbf{j}}
\newcommand{\bW}{\mathbf{W}}

\begin{document}
\baselineskip=1.5pc

\vspace{.5in}

\begin{center}

{\large\bf An adaptive multiresolution interior penalty discontinuous Galerkin method for wave equations in second order form}

\end{center}

\vspace{.1in}

\centerline{
Juntao Huang \footnote{Department of Mathematics,
Michigan State University, East Lansing, MI 48824, USA.
E-mail: huangj75@msu.edu} \qquad
Yuan Liu\footnote{Department of Mathematics, Statistics and Physics,
Wichita State University, Wichita, KS 67260, USA.
E-mail: liu@math.wichita.edu.
Research supported in part by a grant from the Simons Foundation (426993, Yuan Liu). }\qquad
Wei Guo  \footnote{Department of Mathematics and Statistics, Texas Tech University, Lubbock, TX, 70409. E-mail:
weimath.guo@ttu.edu. Research is supported by NSF grant DMS-1830838} \qquad
Zhanjing Tao  \footnote{School of Mathematics, Jilin University, Changchun, Jilin 130012, China. zjtao@jlu.edu.cn. Corresponding author}  \qquad
Yingda Cheng  \footnote{Department of Mathematics, Department of Computational Mathematics, Science and Engineering, Michigan State University, East Lansing, MI 48824, USA. E-mail: ycheng@msu.edu. Research is supported by NSF grants DMS-1453661 and DMS-1720023}
}

\vspace{.4in}

\centerline{\bf Abstract}

In this paper, we propose a class of adaptive multiresolution (also called adaptive sparse grid) discontinuous Galerkin (DG) methods for simulating scalar wave equations in second order form in space. The two key ingredients of the schemes include an interior penalty DG formulation   in the adaptive function space and two classes of multiwavelets for achieving multiresolution. In particular, the orthonormal Alpert's multiwavelets are used to express the DG solution in terms of a hierarchical structure, and the interpolatory multiwavelets are further introduced to enhance computational efficiency in the presence of variable wave speed or nonlinear source. Some theoretical results on stability and accuracy of the proposed method are presented. Benchmark numerical tests in 2D and 3D are provided to validate the performance of the method.

\bigskip

\bigskip

{\bf Key Words:}
Sparse grid; Multiresolution; Interior Penalty Discontinuous Galerkin Method; Wave Equation;  Adaptivity.

\pagenumbering{arabic}

\section{Introduction}
\label{sec1}
\setcounter{equation}{0}
\setcounter{figure}{0}
\setcounter{table}{0}

Wave propagation, governed by the wave equation, is ubiquitous in science and engineering, such as sound waves, light waves, and water waves propagating in acoustics, electromagnetics and geoscience. Designing efficient and robust numerical methods to solve the wave equation is of fundamental and practical importance in those applications. The goal of this work is to design a class of numerical solvers that are adaptive, high order accurate, and more importantly, work efficiently in high dimensions.  In particular,   we develop a class of adaptive multiresolution (also called adaptive sparse grid) discontinuous Galerkin (DG) method for the following model second-order wave equation  
\begin{equation}\label{eq1}
u_{tt} =  \nabla \cdot (c^2(\mathbf{x}) \nabla u ) + f
\end{equation}
on the bounded domain $\Omega=[0, 1]^d$ in arbitrary $d$ dimensions, subject to initial conditions  
\begin{equation}\label{eq2}
u(\mathbf{x}, 0) = u_0(\mathbf{x}), \qquad u_t(\mathbf{x}, 0) = v_0(\mathbf{x}).
\end{equation} 
We assume that the wave speed  $c(\mathbf{x})$ is piecewise smooth and  bounded below and above uniformly, i.e., $0<C_* \le c^2(\bx) \le C^* < \infty$ . For simplicity, we only consider periodic, Dirichlet and Neumann boundary conditions in this paper. Extensions to more complicated domains and other types of boundary conditions will be considered in the future work.

A vast amount of numerical methods have been developed in the literature on the numerical approximations of the wave equation, including finite difference discretization \cite{gustafsson1995time,Cohen_2002,henshaw:1730,sjogreen2012fourth}, spectral and spectral element discretization \cite{gottlieb1977numerical,seriani1994spectral,Tromp:2008zr} and finite element discretization \cite{Jol-2003,ainsworth2006dispersive}, to name a few.  As a special class of finite element discretization, the DG methods \cite{Reed_hill_73,Cockburn_2000_history} have become very popular recently in approximating partial differential equations (PDEs) due to their distinguished advantages in handling  geometry, boundary conditions and accommodating adaptivity. In the context of the wave simulations,  DG methods have been successfully developed for simulating wave equations in first-order form \cite{HesthavenWarburton02,kaser2006arbitrary,wilcox2010high},   second-order form \cite{grote2006discontinuous,xing2013energy, chou2014wave,appelo2015new}, and with $hp$-adaptivity \cite{etienne2010hp}. In this paper, we utilize the symmetric interior penalty DG (IPDG) method \cite{arnold1982interior} for wave equation in second order form \cite{grote2006discontinuous}, though our framework can work with other types of DG schemes. 

 Adaptivity is crucial for efficient simulations of the wave equation due to the multiscale nature of the solution structures. The well-known adaptive mesh refinement (AMR) \cite{BERGER1984484,berger1989local}  adjusts the computational grid adaptively to track small scale features of the underlying problems, improving computational efficiency significantly.
  AMR has been incorporated in various software framework and packages to simulate wave propagation with great success \cite{brown1997overture,p4est}.
 In contrast, this paper considers adaptive simulations in the multiresolution sense. The main idea of multiresolution analysis (MRA) \cite{mallat1999wavelet} is to explore mesh hierarchy, which induces nested polynomial approximation spaces to accelerate the computation and in the mean time circumvents the need for \emph{a posteriori} error indicators. MRA is also the foundation of sparse grid methods \cite{bungartz2004sparse}, which is known as a popular dimension reduction technique for solving high dimensional problems.
As a continuation of our previous research for adaptive multiresolution (also called adaptive sparse grid) DG methods \cite{guo2017adaptive, huang2019adaptive} for first order equations, this paper develops an adaptive multiresolution IPDG solver for 2D and 3D scalar wave equations \eqref{eq1}. In particular, we employ the Alpert's multiwavelets as the DG bases in the IPDG formulation, following the approach proposed in \cite{wang2016elliptic,guo2016transport,guo2017adaptive} for linear equations, together with the interpolatory multiwavelets  for efficiently computing variable wave speed problems as done in \cite{huang2019adaptive} for nonlinear hyperbolic conservation laws. We refer the readers to \cite{huang2019adaptive} for more details on the background of adaptive multiresolution DG methods \cite{calle_wavelets_2005, hovhannisyan2014scalar}. It is worth noting that a fast matrix-vector multiplication algorithm \cite{shen2010efficient,zeiser2011fast} is essential for efficient implementation of the method with varying wave speed.  
We conducts error analysis for the semi-discrete formulation for the scheme with and without interpolations. First, when the sparse grid piecewise polynomial space of degree $k$ is  employed in the IPDG formulation as in \cite{wang2016elliptic}, the newly proposed method converges with order $k$ and a polylogarithmic  factor in the energy norm for sufficiently smooth problems with constant coefficients. Second,  in the case of smooth problems with variable coefficients, the proposed interpolatory technique ensures a high order local truncation error and hence preserve the original accuracy of the scheme given sufficient high order accuracy of interpolation. Numerical experiments in 2D and 3D verify the accuracy of the methods. In particular, the adaptive scheme is demonstrated to capture the fine scale structure presented in inhomogeneous media.

The rest of the paper is organized as follows. In Section \ref{sec:mra}, we review Alpert's and interpolatory multiwavelets.
Section \ref{sec:dg} describes the numerical schemes with details on some theoretical results and implementations.
Section \ref{sec:numeric} contains numerical examples.
In Section \ref{sec:conclusion}, we make conclusions and discuss future work. 
Appendix collects detailed formulas of interpolatory multiwavelets used in this paper.

\section{MRA and multiwavelets}
\label{sec:mra}

In this section, we first review the $L^2$ orthonormal Alpert's multiwavelets \cite{alpert1993class} and the sparse grid DG finite element space \cite{wang2016elliptic,guo2016transport}. Next, we review the interpolatory multiwavelets  proposed in \cite{tao2019collocation}, which has been used for the calculation of nonlinear conservation laws in \cite{huang2019adaptive}.

\subsection{Alpert's multiwavelets}\label{subsec:alpert-basis}

In this subsection, we review the construction of sparse grid DG finite element space based on Alpert's multiwavelets \cite{alpert1993class}. For a unit domain $\Omega=[0,1]$ in 1D,  we define a set of nested grids, where the $n$-th level grid $\Omega_n$ consists of $2^n$ uniform cells
\begin{equation*}
  I_{n}^j=(2^{-n}j, 2^{-n}(j+1)], \quad j=0, \ldots, 2^n-1
\end{equation*}
for $n \ge 0.$ For notational convenience, we also denote $I_{-1}=[0,1].$
The piecewise polynomial space of degree at most $k\ge1$ on the $n$-th level grid $\Omega_n$ for $n\ge 0$ is denoted by
\begin{equation}\label{eq:DG-space-Vn}
V_n^k:=\{v: v \in P^k(I_{n}^j),\, \forall \,j=0, \ldots, 2^n-1\}.
\end{equation}
Because of the nested structure
$$V_0^k \subset V_1^k \subset V_2^k \subset V_3^k \subset  \cdots,$$
we define the multiwavelet subspace $W_n^k$, $n=1, 2, \ldots $ as the orthogonal complement of $V_{n-1}^k$ in $V_{n}^k$ with respect to the $L^2$ inner product on $[0,1]$, i.e.,
\begin{equation*}
V_{n-1}^k \oplus W_n^k=V_{n}^k, \quad W_n^k \perp V_{n-1}^k.
\end{equation*}
Denote $W_0^k:=V_0^k$, we have $V_n^k=\bigoplus_{0 \leq l \leq n} W_l^k$. 
A set of orthonormal basis can be defined on $W_l^k$ as follows. When $l=0$, the basis $v^0_{i,0}(x)$, $ i=0,\ldots,k$ are the normalized shifted Legendre polynomials in $[0,1]$. When $l>0$, the Alpert's orthonormal multiwavelets are employed \cite{alpert1993class} as the bases and denoted by 
$$v^j_{i,l}(x),\quad i=0,\ldots,k,\quad j=0,\ldots,2^{l-1}-1.$$

We then follow a tensor-product approach to construct the hierarchical finite element space in multi-dimensional space.  Denote $\bl=(l_1,\cdots,l_d)\in\mathbb{N}_0^d$ as the mesh level in a multivariate sense, where $\mathbb{N}_0$  denotes the set of nonnegative integers, we can define the tensor-product mesh grid $\Omega_\bl=\Omega_{l_1}\otimes\cdots\otimes\Omega_{l_d}$ and the corresponding mesh size $h_\bl=(h_{l_1},\cdots,h_{l_d}).$ Based on the grid $\Omega_\bl$, we denote  $I_\bl^\bj=\{\bx:x_m\in(h_mj_m,h_m(j_{m}+1)),m=1,\cdots,d\}$ as an elementary cell, and 
$$\bV_\bl^k:=\{\bv:  \bv \in Q^k(I^{\bj}_{\bl}), \,\,  \bzero \leq \bj  \leq 2^{\bl}-\bone \}= V_{l_1,x_1}^k\times\cdots\times  V_{l_d,x_d}^k$$
as the tensor-product piecewise polynomial space, where $Q^k(I^{\bj}_{\bl})$ represents the collection of polynomials of degree up to $k$ in each dimension on cell $I^{\bj}_{\bl}$. 
If we use equal mesh refinement of size $h_N=2^{-N}$ in each coordinate direction, the  grid and space will be denoted by $\Omega_N$ and $\bV_N^k$, respectively.  
Based on a tensor-product construction, the multi-dimensional increment space can be  defined as
$$\bW_\bl^k=W_{l_1,x_1}^k\times\cdots\times  W_{l_d,x_d}^k.$$
The basis functions in multi-dimensions are defined as
\begin{equation}\label{eq:multidim-basis}
v^\bj_{\bi,\bl}(\bx) := \prod_{m=1}^d v^{j_m}_{i_m,l_m}(x_m),
\end{equation}
for $\bl \in \mathbb{N}_0^d$, $\bj \in B_\bl := \{\bj\in\mathbb{N}_0^d: \,\mathbf{0}\leq\bj\leq\max(2^{\bl-\mathbf{1}}-\mathbf{1},\mathbf{0}) \}$ and $\mathbf{1}\leq\bi\leq \bk+\mathbf{1}$. The orthonormality of the bases can be easily verified.

Using the notation of $$
|\bl|_1:=\sum_{m=1}^d l_m, \qquad   |\bl|_\infty:=\max_{1\leq m \leq d} l_m.
$$
and  the same component-wise arithmetic operations and relations   as defined in \cite{wang2016elliptic},  we reach the decomposition
\begin{equation}\label{eq:hiere_tp}
\bV_N^k=\bigoplus_{\substack{ |\bl|_\infty \leq N\\\bl \in \mathbb{N}_0^d}} \bW_\bl^k.
\end{equation}
On the other hand, a standard choice of sparse grid   space  \cite{wang2016elliptic, guo2016transport} is
\begin{equation}
\label{eq:hiere_sg}
\hat{\bV}_N^k=\bigoplus_{\substack{ |\bl|_1 \leq N\\\bl \in \mathbb{N}_0^d}}\bW_\bl^k \subset \bV_N^k.
\end{equation}
We skip the discussions on the details with regard to the property of the space, but refer the readers to \cite{wang2016elliptic, guo2016transport}. In Section \ref{sec:dg}, we will describe   the adaptive scheme which adapts a subspace of $\bV_N^k$ according to the numerical solution, hence offering more flexibility and efficiency.

\subsection{Interpolatory multiwavelets}\label{subsec:interp-basis}

Alpert's multiwavelets described in Section \ref{subsec:alpert-basis} are associated with the $L^2$ projection operator. The idea of interpolatory multiwavelet bases   \cite{tao2019collocation} is based on interpolation operators and is essential for the computation of variable coefficient problems. In this work, only Lagrange interpolation is considered, while we note that  Hermite interpolation can be used. The details are provided below.

 We define the set of interpolation points on the interval $I=[0,1]$ at mesh level 0 by $X_0 = \{ x_i \}_{i=0}^M\subset I$. Here, the number of points in $X_0$ is $(M+1)$. We defer the discussion of the relations between $M$ and $k$ to Section \ref{sec:schemeinter}.
 
  The interpolation points at mesh level $n\ge1$, $X_n$ can be obtained correspondingly as
\begin{equation*}
    X_n = \{ x_{i,n}^j := 2^{-n}(x_i+j), \quad i=0,\dots,M, \quad j=0,\dots,2^{n}-1 \}.
\end{equation*}
We require the points to be nested, i.e. 
\begin{equation}
\label{nestpts}
    X_0 \subset X_1 \subset X_2 \subset X_3 \subset \cdots.
\end{equation}
This can be achieved by requiring  $X_0\subset X_1$.

Given the nodes, we define the basis functions on the zeroth level grid as  Lagrange interpolation polynomials of degree $\le M$ which satisfy the property:
\begin{equation*}
    \phi_{i}(x_{i'}) = \delta_{ii'},
\end{equation*}
for $ i,i'=0,\dots,M$. It is easy to see that $\textrm{span} \{ \phi_{i},  i=0,\dots,M \}=V_0^M.  $ 
With the basis function at mesh level zero, we can define the basis functions at mesh level $n\ge1$:
\begin{equation*}
  \phi_{i,n}^j := \phi_{i}(2^nx-j), \quad i=0,\dots,M, \quad j=0,\dots,2^n-1 
\end{equation*}
which form a complete basis set for $V_n^M.$

We now introduce the hierarchical representations and the interpolatory multiwavelets. Define $\tilde{X}_0 := X_0$ and { $\tilde{X}_n := X_n\backslash X_{n-1}$} for $n\ge1$, then we have the decomposition
\begin{equation*}
    X_n = \tilde{X}_0 \cup \tilde{X}_1 \cup \cdots \cup \tilde{X}_n .
\end{equation*}
Denote the points in $\tilde{X}_1$ by $\tilde{X}_1=\{ \tilde{x}_i \}_{i=0}^M$. Then the points in $\tilde{X}_n$ for $n\ge1$ can be represented by
\begin{equation*}
    \tilde{X}_n = \{ \tilde{x}_{i,n}^j:=2^{-(n-1)}(\tilde{x}_i+j), \quad  i=0,\dots,M, \quad j=0,\dots,2^{n-1}-1 \}.
\end{equation*}


For notational convenience, we let $\tilde{W}_0^M:=V_0^M.$ The increment function space $\tilde{W}_n^M$ for $n\ge1$ is introduced as a function space    that satisfies
\begin{equation}\label{eq:func-space-sum}
    V_n^M = V_{n-1}^M \oplus \tilde{W}_n^M,
\end{equation}
and is defined through the multiwavelets
 $\psi_{i} \in V_1^M$ that satisfies
\begin{equation*}
    \psi_{i}(x_{i'}) = 0, \quad \psi_{i}(\tilde{x}_{i'}) = \delta_{i,i'},
\end{equation*}
for $i,i'=0,\dots,M$. Then  $\tilde{W}_n^M$ is given by
\begin{equation*}
    \tilde{W}_n^M = \textrm{span} \{ \psi_{i,n}^j, \quad i = 0,\dots,M, \quad j=0,\dots,2^{n-1}-1 \}
\end{equation*}
where $\psi_{i,n}^j(x) := \psi_{i}(2^{n-1}x-j)$.


The multi-dimensional construction follows similar lines as in Section \ref{subsec:alpert-basis}. We let 
$$\tilde{\bW}_\bl^M=\tilde{W}_{l_1,x_1}^M\times\cdots\times  \tilde{W}_{l_d,x_d}^M,$$
then 
\begin{equation*}
\bV_N^M=\bigoplus_{\substack{ |\bl|_\infty \leq N\\\bl \in \mathbb{N}_0^d}} \tilde{\bW}_\bl^M,
\end{equation*}
while the sparse grid approximation space is
\begin{equation*}
\hat{\bV}_N^M=\bigoplus_{\substack{ |\bl|_1 \leq N\\\bl \in \mathbb{N}_0^d}}\tilde{\bW}_\bl^M.
\end{equation*}
Note that the construction by Alpert's multiwavelet and the interpolatory multiwavelet gives the same sparse grid space because of the same nested structure. 
Finally, the interpolation operator in multidimension is defined as $\mathcal{I}^{M}_{N}: C(\Omega)\rightarrow \mathbf{V}^{M}_{N}$:
\begin{align*}
	\mathcal{I}^{M}_{N}[f](\mathbf{x}) 
	= \sum_{ \substack{ \abs{\mathbf{n}}_{\infty}\leq N \\ \mathbf{0}\leq \mathbf{j} \leq \max(2^{\mathbf{n}-1}-\mathbf{1},\mathbf{0}) \\ \mathbf{0}\leq \mathbf{i}\leq \mathbf{M}  } } b^{\mathbf{j}}_{\mathbf{i},  \mathbf{n}} \psi^{\mathbf{j}}_{\mathbf{i},  \mathbf{n}} (\mathbf{x}),
\end{align*}
where the multi-dimensional basis functions $\psi^{\mathbf{j}}_{\mathbf{i}, \mathbf{n}} (\mathbf{x})$ are defined in the same approach as \eqref{eq:multidim-basis} by tensor products:
\begin{equation}\label{eq:multidim-basis-interpolation}
	\psi^{\mathbf{j}}_{\mathbf{i}, \mathbf{n}} (\mathbf{x}) := \prod_{m=1}^d \psi^{j_m}_{i_m,n_m}(x_m).
\end{equation}
For the sparse grid space $\hat{\bV}_N^M$ or any adaptively chosen subspace of $\bV_N^M,$ the interpolation operator, which is denoted by $\mathcal{I}_h$ in later sections, can be defined accordingly, by taking only multiwavelet basis functions that belong to that space. For completeness, we collect the detailed formulas of the interpolation points and the associated interpolatory multiwavelets used in this work in the Appendix.

\section{Adaptive multiresolution DG scheme}
\label{sec:dg}

In this section, we construct  our numerical schemes for $d$-dimensional wave equation \eqref{eq1}.  We start by reviewing the semi-discrete IPDG formulation and its properties in Section \ref{sec:semi1}. For variable wave speed, schemes with multiresolution interpolation are described in Section \ref{sec:schemeinter}. Time stepping, adaptivity and fast implementations are discussed in Sections \ref{subsec:timeadapt} and \ref{subsec:fast}. 

\subsection{Semi-discrete scheme}
\label{sec:semi1}

We use the IPDG formulation \cite{grote2006discontinuous} for solving \eqref{eq1}.    Namely, we look for $u_h \in \bV,$ such that for any test function $v \in \bV$,
\begin{align}\label{sp2}
\int_{\Omega} (u_h)_{tt} v \ d\textbf{x}  +  B(u_h, v)  =L(v).
\end{align}
where the bilinear form is defined as 
\begin{align}\label{sp3}
B(u_h, v) & = \int_{\Omega} c^2 \nabla u_h \cdot \nabla v \ d\textbf{x}  - \sum_{e\in \Gamma } \int_{e} \{c^2 \nabla u_h \} \cdot [v] \ ds  -\sum_{e \in \Gamma} \int_{e} \{c^2\nabla v \} \cdot [u_h] \ ds  \\ \nonumber & +  \sum_{e \in \Gamma} \frac{\sigma}{h_N} \int_e [u_h] \cdot [v] \ ds
\end{align}
and
\begin{equation}\label{sp4}
L(v) =  \int_{\Omega} fv \ d \textbf{x}
\end{equation}
for periodic or homogeneous Dirichlet boundary condition, and 
\begin{equation}\label{sp5}
L(v) = \int_{\Omega} fv d \textbf{x}
+ \sum_{e\in  \Gamma_D} \int_e (-c^2\nabla v \cdot \textbf{n} +\frac{\sigma}{h_N}v)g_D ds + \sum_{e\in  \Gamma_N} \int_e g_Nv ds
\end{equation}
for  Dirichlet and Neumann boundary conditions $u(\mathbf{x}, t)|_{\mathbf{x} \in \Gamma_D} = g_D$ and $\nabla u(\mathbf{x}, t) \cdot \textbf{n}|_{\mathbf{x} \in \Gamma_N} = g_N$.
$\Gamma $ is the union of the boundaries for all the elements in the partition $\Omega_N$, and $\sigma$ is the penalty parameter depending on the dimension $d.$ The average and jump are defined as,
\begin{align}\label{sp5}
[q]=q^- \textbf{n}^- + q^+ \textbf{n}^+, \qquad & \{ q\} = \frac{1}{2} (q^-+q^+),  \nonumber \\
[\textbf{q}] = \textbf{q}^- \cdot \textbf{n}^-  + \textbf{q}^+ \cdot \textbf{n}^+ ,  \qquad & \{ \textbf{q}\} = \frac{1}{2}(\textbf{q}^- + \textbf{q}^+).
\end{align}
where $\textbf{n}$ is the unit normal. `-' and `+' represent that the directions of the vector point to interior and exterior at $e$ respectively.  If $e$ is part of the boundary, then we let $[q] = q \textbf{n}$ ($\textbf{n}$ is the outward unit normal) and $\{\textbf{q} \} = \textbf{q}$.

Depending on the choice of space $\bV$, various IPDG methods with distinct properties are obtained.  If   $\bV=\bV^k_N,$ we recover the IPDG scheme in \cite{grote2006discontinuous} on tensor-product meshes. If $\bV=\hat{\bV}^k_N,$ then we obtain  the sparse grid IPDG method. If $\bV$ is chosen adaptively as described in Section \ref{subsec:timeadapt}, we have the adaptive multiresolution scheme. Note that  besides the IPDG formulation,  other DG formulations can be used as well, such as the local DG method \cite{chou2014wave} and the energy-based DG method \cite{appelo2015new}. The main novelty of this work is the choice of the multiresolution polynomial space which is not tied specifically to the  weak formulation in use.

\medskip

For completeness, we now review some properties of the semi-discrete IPDG scheme (\ref{sp2}). Define the discrete energy of wave propagation  by
\begin{align}\label{se2}
E_h(t) := \frac{1}{2} \left \Vert \frac{\partial u_h}{\partial t}\right \Vert^2 +\frac{1}{2} B(u_h, u_h),
\end{align}
Then the stability inherently holds true since the bilinear form $B(\cdot, \cdot)$ is symmetric and coercive:
\begin{thm}[Energy stability  \cite{grote2006discontinuous}]
\label{thm:stable}
The discrete energy (\ref{se2}) is conserved by semi-discrete DG scheme (\ref{sp2})-(\ref{sp4}) when $f=0$ with periodic boundary condition for arbitrary choice of space including $\bV=\bV^k_N$ and $\bV=\hat{\bV}^k_N$.
\end{thm}

We then review some results in the error estimates  \cite{grote2006discontinuous}, and extend it to the sparse grid method with   $\bV=\hat{\bV}^k_N$ based on the approximation properties of the space $\hat{\bV}^k_N$ in \cite{guo2016transport}.   We use $\Vert \cdot \Vert$ to represent the standard $L^2$ norm on $\Omega$ or $\Omega_N,$ $\Vert \cdot \Vert_{L^2(\Gamma)}$ to represent the $L^2$ norm on the collection of the cell interfaces of the mesh $\Omega_N: \Gamma,$ and define the energy norm of a function $v\in H^2(\Omega_N)$ as
\begin{align}\label{se1}
 ||| v ||| ^2 :=  \int_{\Omega} |\nabla v|^2 \,d\bx \, + \sum_{\substack{e \in \Gamma}}h_N \int_{e} \left \{  \frac{\partial v}{\partial \bn} \right \}^2\,ds\, + \sum_{\substack{e \in \Gamma}}\frac{1}{h_N} \int_{e} [v]^2\,ds.
\end{align}

Some basic properties of the bilinear operator $B(\cdot, \cdot)$ are listed below.

\begin{lem}[Boundedness \cite{arnold1982interior, arnold2002unified}]
\label{lem:bound}
There exists a positive constant $C_b$, depending only on $C^*, \sigma$, such that
$$|B( w, v )| \le C_b |||w |||\cdot ||| v |||,   \quad \forall \,w, v \in H^2(\Omega_N).$$
\end{lem}

\begin{lem}[Coercivity \cite{arnold1982interior, arnold2002unified}]
\label{lem:stab}
When $\sigma$ is taken large enough, there exists a positive constant $C_s$ depending only on $C_*$, such that
$$B( v, v ) \ge C_s  ||| v |||^2,   \quad \forall \,  v \in \hat{\bV}_N^k.$$
\end{lem}

  Then we arrive at the following error estimate.
 \begin{thm}
[Error estimate in energy norm for sparse grid IPDG method] 
\label{thm:error}
Let $u$ be the solution of (\ref{eq1})-(\ref{eq2}) satisfying $u \in L^{\infty}(0, T; \mathcal{H}^{p+1}(\Omega))$, $u_t \in L^{\infty}(0, T; \mathcal{H}^{p+1}(\Omega))$, $u_{tt} \in L^1(0, T; \mathcal{H}^p(\Omega)).$ 
$u_h$ is the semi-discrete DG solution obtained by (\ref{sp2})-(\ref{sp4}) with $\bV=\hat{\bV}^k_N$ and the initial condition $u_h(0) = \mathbf{P}u_0$ and $(u_h)_t(0) = \mathbf{P}v_0$, where $\mathbf{P}$ denotes the $L^2$ projection of a function onto the space $\hat{\bV}^k_N$. Then for $k \ge 1$ and any $1 \le q \le \min\{p, k\} $, the error $e = u_h - u$ satisfies the estimation
\begin{align}
||e_t ||_{L^{\infty}(0, T; L^2(\Omega))} & + \sup_{t\in [0, T]}|||e ||| \le C (||e_t(0)|| + ||| e(0)|||)  \\\nonumber &+  C\abs{\log_2h_N}^dh_N^{q} \left(|u|_{L^{\infty}(0, T; \mathcal{H}^{q+1}(\Omega))} + T |u_t|_{L^{\infty}(0, T; \mathcal{H}^{q+1}(\Omega))}  + |u_{tt}|_{L^{\infty}(0, T; \mathcal{H}^{q}(\Omega))}  \right)
\end{align}  
where the dimension $d\ge 2.$ $|\cdot|_{\mathcal{H}^{q+1}(\Omega)}$ denotes mixed derivative norm of a function and was defined in \cite{guo2016transport}.  Here and below, $C$ denotes a generic constant that   does not depend on $h_N$ or the solution $u.$
\end{thm}
\begin{proof} Following \cite{grote2006discontinuous}, we   let $\bV(h) = H^1(\Omega) + \hat{\bV}^k_N, $ and for any $v \in \bV(h),$ we  define the lifted function $\mathcal{L}_c(v) \in (\hat{\bV}^k_N)^d $ by requiring 
\begin{align}\label{en000_0}
\int_{\Omega} \mathcal{L}_c(v) \cdot w d \textbf{x} = \sum_{e \in \Gamma} \int_e [v] \cdot \{c^2w \} ds, \qquad w\in (\hat{\bV}^k_N)^d.
\end{align}
Using similar arguments as in Lemma 4.3 in \cite{grote2006discontinuous}, we conclude the lifting operator $\mathcal{L}_c$ exists and is stable in the DG norm. Then the auxiliary bilinear form can be introduced as
\begin{align}\label{en0_0}
\hat{B}(u, v) & = \int_{\Omega} c^2 \nabla u \cdot \nabla v \ d\textbf{x}  - \int_{\Omega}\mathcal{L}_c(u) \cdot \nabla{v} \ ds  - \int_{\Omega} \mathcal{L}_c(v) \cdot \nabla{u} \ ds  \\ \nonumber & +  \sum_{e \in \Gamma} \frac{\sigma}{h_N} \int_e [u] \cdot [v] \ ds.
\end{align}
$\hat{B}(u, v)$ can be viewed as an extension of the wave operator and bilinear form $B(u, v)$ to the space $\bV(h) \times \bV(h)$, since
\begin{align}\label{en00_0}
 & \hat{B}(u, v)  =  B(u, v) \quad  \text{on} \quad \hat{\bV}^k_N \times \hat{\bV}^k_N, \\
 & \hat{B}(u, v)  =  \int_{\Omega} c^2 \nabla u \cdot \nabla v d\textbf{x} - \int_{\partial \Omega} (c^2 \nabla u) v\cdot \bn ds \quad  \text{on} \quad H^1(\Omega) \times H^1(\Omega). 
\end{align}
Moreover, it can be verified that  
 \begin{align}\label{en3_0}
& \hat{B}(u, v) \le C_b ||| u ||| \cdot ||| v |||,  \\ \nonumber
& \hat{B}(u, u) \ge C_s ||| u |||^2.
\end{align}
Similar to Lemma 4.5 in \cite{grote2006discontinuous}, $e$ satisfies the equation 
\begin{align}\label{en0}
(e_{tt}, v) + \hat{B}(e, v) = r_h(u,v), \quad \forall v \in \hat{\bV}^k_N
\end{align}
where
\begin{align}\label{en0_01}
 r_h(u, v) = \sum_{e \in \Gamma} \int_e [v] \cdot \{c^2\nabla u -c^2 \mathbf{P}(\nabla u) \} ds. 
\end{align}
Therefore, we will have
\begin{align}\label{en1}
\frac{1}{2} \frac{d}{dt} [||e_t||^2 +\hat{B}(e, e)] & = (e_{tt}, e_t) + \hat{B}(e, e_t) \\ \nonumber
& = (e_{tt}, (u-\mathbf{P}u)_t) + \hat{B}(e, (u-\mathbf{P}u)_t) + r_h(u, (\mathbf{P}u-u_h)_t).
\end{align}
 Integrating (\ref{en1}) over $[0, s]$ for any $s \in [0, T]$ yields
 \begin{align}\label{en2}
 \frac{1}{2}||e_t(s)||^2 & +\frac{1}{2}\hat{B}(e(s), e(s))  = \frac{1}{2}||e_t(0)||^2 +\frac{1}{2}\hat{B}(e(0), e(0)) + \int^s_0 (e_{tt}, (u-\mathbf{P}u)_t) dt \\ \nonumber &
                             +\int^s_0 \hat{B}(e, (u-\mathbf{P}u)_t) dt + \int^s_0 r_h(u, (\mathbf{P}u-u_h)_t) dt. 
 \end{align} 
 Because 
 \begin{align}\label{en3}
 \int^s_0 (e_{tt}, (u-\mathbf{P}u)_t) dt  = - \int^s_0 (e_t, (u-\mathbf{P}u)_{tt}) dt + [(e_t, (u-\mathbf{P}u)_t)]^{t=s}_{t=0},
\end{align} 
and the inequalities (\ref{en3_0}) hold,
 together with Holder's inequalities, we will have  
 \begin{align}\label{en5}
 \frac{1}{2}||e_t(s)||^2  +\frac{1}{2} C_s|||e(s)|||^2  & \le \frac{1}{2}||e_t(0)||^2 +\frac{1}{2} C_b ||| e(0)|||^2  \\ \nonumber  
 &+ || e_t||_{L^{\infty}(0, T; L^2(\Omega))}( ||(u-\mathbf{P}u)_{tt}||_{L^1(0, T; L^2(\Omega))}  + 2||(u-\mathbf{P}u)_t ||_{L^{\infty}(0, T; L^2(\Omega))})  \\ \nonumber
&  + C_b T |||e ||| \cdot  |||(u-\mathbf{P}u)_t||| \\ \nonumber
& + \left|\int^T_0 r_h(u, (\mathbf{P}u-u_h)_t) dt \right|.
 \end{align}
Since the inequality (\ref{en5}) holds for any $s \in [0, T]$, taking the maximum on $[0, T]$ will result in  
\begin{align}\label{en6}
||e_t||^2_{L^{\infty}(0, T; L^2(\Omega))} +  C_s ||e||^2_{L^{\infty}(0, T; \bV(h))}  \le ||e_t(0)||^2 + C_b ||| e(0)|||^2 + T_1 +T_2 +T_3
\end{align}
where the short-hand notation $||e||_{L^{\infty}(0, T; \bV(h))} :=\sup_{t\in[0, T]} |||e |||$ is introduced, and
\begin{align}\label{en7}
& T_1 = 2 || e_t||_{L^{\infty}(0, T; L^2(\Omega))}( ||(u-\mathbf{P}u)_{tt}||_{L^1(0, T; L^2(\Omega))} + 2||(u-\mathbf{P}u)_t ||_{L^{\infty}(0, T; L^2(\Omega))}) \\ \nonumber
& T_2 = 2C_b T  |||e ||| \cdot   |||(u-\mathbf{P}u)_t||| \\ \nonumber
& T_3 = 2  \left|\int^T_0 r_h(u, (\mathbf{P}u-u_h)_t) dt \right|.
\end{align}
Using the geometric-arithmetic mean inequality, and Lemma 3.2 in \cite{guo2016transport}, we conclude
\begin{align}\label{en8}
 T_1 & \le \frac{1}{2} ||e_t||^2_{L^{\infty}(0, T; L^2(\Omega))} + 2(|| (u-\mathbf{P}u)_{tt} ||_{L^1(0, T; L^2(\Omega))} + 2||(u-\mathbf{P}u)_t ||_{L^{\infty}(0, T; L^2(\Omega))})^2 \\ \nonumber
        & \le  \frac{1}{2} ||e_t||^2_{L^{\infty}(0, T; L^2(\Omega))} + 4 || (u-\mathbf{P}u)_{tt} ||_{L^1(0, T; L^2(\Omega))}^2 + 16 ||(u - \mathbf{P}u)_t ||_{L^{\infty}(0, T; L^2(\Omega))}^2 \\ \nonumber
        & \le  \frac{1}{2} ||e_t||^2_{L^{\infty}(0, T; L^2(\Omega))} + C \abs{\log_2h_N}^{2d}h_N^{2q}(|u_{tt}|^2_{L^\infty(0, T; \mathcal{H}^{q}(\Omega))} +h_N^2 |u_t|^2_{L^2(0, T;  \mathcal{H}^{q+1}(\Omega))}). 
\end{align}
Similarly
\begin{align}\label{en9}
T_2 & \le \frac{1}{4}C_s ||| e |||^2 + 4\frac{C^2_b}{C_s} T^2  ||| (u-\mathbf{P}u)_t |||^2 \\ \nonumber
        & \le \frac{1}{4}C_s ||e||_{L^{\infty}(0, T; \bV(h))} ^2 + CT^2|\log_2h_N|^{2d}h_N^{2q+2} |u_t|^2_{L^2(0, T; \mathcal{H}^{q+1}(\Omega))}.
\end{align}
We then start to bound the term $T_3$. From (\ref{en0_01}), we can derive 
\begin{align}\label{en10}
|r_h(u, v) | & = |\sum_{e \in \Gamma} \int_e [v] \cdot \{ c^2 \nabla u - c^2 \mathbf{P}(\nabla u) \} ds | \\ \nonumber
                 & \le (\sum_{e \in \Gamma} \int_e \frac{\sigma}{h_N} [v]^2 ds )^{\frac{1}{2}} (\sum_{e \in \Gamma } \int_e  \frac{h_N}{\sigma} | c^2\nabla u - c^2 \mathbf{P}(\nabla u) |^2 ds)^{\frac{1}{2}} \\ \nonumber
                 & \le C |||v||| (\sum_{K \in \Omega_N} h_N || \nabla u - \mathbf{P}(\nabla u)  ||_{\partial K}^2)^{\frac{1}{2}},
\end{align}
with $h_N = \frac{1}{2^N}$,  using trace inequality and  Lemma 3.2 in \cite{guo2016transport}, we have
\begin{align}\label{en10_1}
|r_h(u, v) |  \le C||v||_{L^{\infty}(0, T; \bV(h))} \cdot \abs{\log_2h_N}^d h_N^{q} |u|_{L^{\infty}(0, T; \mathcal{H}^{q+1}(\Omega))}.
\end{align}

Therefore, 
\begin{align}\label{en11}
\left|\int^T_0 r_h(u, v_t) \right| & =   \left| -\int^T_0 r_h(u_t, v) dt + r_h(u,v) |^{t=T}_{t=0} \right| \\ \nonumber
& \le C T || v ||_{L^{\infty}(0, T; \bV(h))} \abs{\log_2h_N}^d h_N^{q} |u_t|_{L^{\infty}(0,T; \mathcal{H}^{q+1}(\Omega))} \\ \nonumber 
&+ 2C || v ||_{L^{\infty}(0, T; \bV(h))} \abs{\log_2h_N}^d h_N^{q} |u|_{L^{\infty}(0,T; \mathcal{H}^{q+1}(\Omega))}.
\end{align}
Denote $\mathcal{R} =T |u_t|_{L^{\infty}(0,T; \mathcal{H}^{q+1}(\Omega))} + 2 |u|_{L^{\infty}(0,T; \mathcal{H}^{q+1}(\Omega))}$, we will have 
\begin{align}\label{en12}
T_3 & \le 2C \mathcal{R}\abs{\log_2h_N}^d h_N^{q } || \mathbf{P}u-u_h  ||_{L^{\infty}(0, T; \bV(h))} \\ \nonumber
       & \le 2C \mathcal{R} \abs{\log_2h_N}^d h_N^{q} \left[ ||e ||_{L^{\infty}(0, T; \bV(h)) } + || u-\mathbf{P}u ||_{L^{\infty}(0, T;\bV(h)) } \right] \\ \nonumber
       & \le  \frac{1}{4} C_s || e ||^2_{L^{\infty}(0, T; \bV(h))} + C \abs{\log_2h_N}^{2d}h_N^{2q} \left[ |u|^2_{L^{\infty}(0,T; H^{q+1}(\Omega))} +\mathcal{R}^2 \right].
\end{align}
Together with (\ref{en6}) and the estimates for $T_1$, $T_2$ and $T_3$, we arrive at the estimate 
\begin{align}\label{en13}
\frac{1}{2}||e_t||^2_{\infty} & + \frac{1}{2}C_s \sup_{t\in [0, T]}|||e |||^2  \le ||e_t(0)||^2 + C ||| e(0)|||^2 \\ \nonumber &+ C \abs{\log_2h_N}^{2d}h_N^{2q}(|u_{tt}|^2_{L^{\infty}(0, T; \mathcal{H}^{q}(\Omega))}   + T^2 |u_t|^2_{L^{\infty}(0, T; \mathcal{H}^{q+1}(\Omega))} +|u|^2_{L^{\infty}(0, T; \mathcal{H}^{q+1}(\Omega))} ),
\end{align}
and this completes the proof. 
\end{proof}

\subsection{Semi-discrete scheme with multiresolution interpolation}
\label{sec:schemeinter}

To treat variable coefficient case, we follow the idea in \cite{shen2010efficient,huang2019adaptive} and interpolate the functions $c^2u_h$ and $c^2\nabla u_h$ (or $(c^2)^-\nabla u_h$ and $(c^2)^+\nabla u_h$ in the case when $c^2(\bx)$ contains discontinuity on the cell interfaces of $\Omega_N$) by using the multiresolution Lagrange interpolation  discussed in Section \ref{subsec:interp-basis}. For simplicity of discussion, we only focus on the homogeneous Dirichlet boundary condition with no source term. However, similar results can be established for mixed boundary conditions and also with source terms.

We first assume  $c=c(\bx)$ is continuous. In this case, we can reformulate \eqref{sp2} into an equivalent form
\begin{align}
B(u_h, v) & = \int_{\Omega} c^2 \nabla u_h \cdot \nabla v \ d\textbf{x}  - \sum_{e\in \Gamma } \int_{e} \{c^2 \nabla u_h \} \cdot [v] \ ds  -\sum_{e \in \Gamma}\int_{e} \{\nabla v \} \cdot [c^2u_h] \ ds   \nonumber \\
& +  \sum_{e \in \Gamma} \frac{\sigma}{h_N} \int_e [u_h] \cdot [v] \ ds
\end{align}
then the scheme is implemented by the modified operator with interpolation
\begin{align}\label{eq:DG-semi-interp-discontinuous-continuous}
\tilde{B}(u_h, v) & = \int_{\Omega} \mathcal{I}_h(c^2 \nabla u_h) \cdot \nabla v \ d\textbf{x}  - \sum_{e\in \Gamma } \int_{e} \{\mathcal{I}_h(c^2 \nabla u_h) \} \cdot [v] \ ds  -\sum_{e \in \Gamma}\int_{e} \{\nabla v \} \cdot [\mathcal{I}_h(c^2 u_h)] \ ds   \nonumber \\
& +  \sum_{e \in \Gamma} \frac{\sigma}{h_N} \int_e [u_h] \cdot [v] \ ds
\end{align}
Here, $\mathcal{I}_h(\cdot)$ denote the interpolation operator defined in Section \ref{subsec:interp-basis} with interpolation parameter $M$ to be specified later. 

If $c=c(\bx)$ is discontinuous along the cell interface, then some special care has to be taken for the third term $\sum_{e \in \Gamma}\int_{e} \{c^2\nabla v \} \cdot [u_h] \ ds$. We first reformulate it into another form:
\begin{align*}
	& \{c^2 \nabla v \} \cdot [u_h] \\
	={}& \frac{1}{2} \brac{ (c^2\nabla v)^- + (c^2\nabla v)^+ }\cdot(u_h^-\textbf{n}^- + u_h^+\textbf{n}^+)		\nonumber \\ 
	={}& \frac{1}{2} \brac{ (c^2\nabla v)^-\cdot u_h^-\textbf{n}^-  + (c^2\nabla v)^-\cdot u_h^+\textbf{n}^+  + (c^2\nabla v)^+\cdot u_h^-\textbf{n}^- + (c^2\nabla v)^+\cdot u_h^+\textbf{n}^+}	  \\ \nonumber
	={}& \frac{1}{2}\brac{ (c^2)^-u_h^-\textbf{n}^- \cdot (\nabla v)^- + (c^2)^-u_h^+\textbf{n}^+ \cdot (\nabla v)^- + (c^2)^+ u_h^-\textbf{n}^- \cdot (\nabla v)^+ + (c^2)^+ u_h^+\textbf{n}^+ \cdot (\nabla v)^+ } \\ \nonumber
	={}& \frac{1}{2}((c^2)^-u_h^-\textbf{n}^- + (c^2)^-u_h^+\textbf{n}^+) \cdot (\nabla v)^-  + \frac{1}{2} ((c^2)^+ u_h^-\textbf{n}^- + (c^2)^+ u_h^+\textbf{n}^+) \cdot (\nabla v)^+ \\
	={}& \frac{1}{2}[(c^2)^- u_h]\cdot(\nabla v)^- + \frac{1}{2}[(c^2)^+  u_h]\cdot(\nabla v)^+.
\end{align*}
Here $[(c^2)^- u_h]:=((c^2)^-u_h^-\textbf{n}^- + (c^2)^-u_h^+\textbf{n}^+)$ and $[(c^2)^+ u_h]:=((c^2)^+u_h^-\textbf{n}^- + (c^2)^+u_h^+\textbf{n}^+)$. Now the bilinear form \eqref{sp2} is rewriten into
\begin{align*}
B(u_h, v) & = \int_{\Omega} c^2 \nabla u_h \cdot \nabla v \ d\textbf{x}  - \sum_{e\in \Gamma } \int_{e} \{c^2 \nabla u_h \} \cdot [v] \ ds \\ \nonumber & -\sum_{e \in \Gamma}\int_{e} \brac{\frac{1}{2}[(c^2)^- u_h]\cdot(\nabla v)^- + \frac{1}{2}[(c^2)^+ u_h]\cdot(\nabla v)^+} \ ds   +  \sum_{e \in \Gamma} \frac{\sigma}{h_N} \int_e [u_h] \cdot [v] \ ds
\end{align*}
and then the interpolation operator is performed on $c^2\nabla u_h$, $(c^2)^- u_h$ and also $(c^2)^+ u_h,$ which gives:
\begin{align}\label{eq:DG-semi-interp-discontinuous}
\tilde{B}(u_h, v) & = \int_{\Omega} \mathcal{I}_h(c^2 \nabla u_h) \cdot \nabla v \ d\textbf{x}  - \sum_{e\in \Gamma } \int_{e} \{\mathcal{I}_h(c^2 \nabla u_h) \} \cdot [v] \ ds \\ \nonumber 
& -\sum_{e \in \Gamma}\int_{e} \brac{ \frac{1}{2}[\mathcal{I}_h((c^2)^- u_h)]\cdot(\nabla v)^- + \frac{1}{2}[\mathcal{I}_h((c^2)^+ u_h)]\cdot(\nabla v)^+} \ ds   +  \sum_{e \in \Gamma} \frac{\sigma}{h_N} \int_e [u_h] \cdot [v] \ ds.
\end{align}

Following \cite{DG4}, we can now write the DG scheme with interpolation \eqref{eq:DG-semi-interp-discontinuous} into the semi-discrete form as
\begin{equation}
	\frac{d^2u_h}{dt^2} = L_h(u_h),
\end{equation}
where $L_h(u)$ is an operator onto $\bV$ which is a discrete approximation of $-\nabla \cdot (c^2(\mathbf{x}) \nabla u )$ and satisfies
\begin{align}\label{eq:def-Lh}
	\sum_{K\in {\Omega_N}}\int_{K} L_h(u_h)v_h  \ d\textbf{x}&= -  \int_{\Omega} \mathcal{I}_h(c^2 \nabla u_h) \cdot \nabla v_h \ d\textbf{x}  + \sum_{e\in \Gamma } \int_{e} \{\mathcal{I}_h(c^2 \nabla u_h) \} \cdot [v_h] \ ds \nonumber \\ 
& + \sum_{e \in \Gamma}\int_{e} \brac{ \frac{1}{2}[\mathcal{I}_h((c^2)^- u_h)]\cdot(\nabla v_h)^- + \frac{1}{2}[\mathcal{I}_h((c^2)^+ u_h)]\cdot(\nabla v_h)^+} \ ds   \nonumber \\ 
& - \sum_{e \in \Gamma} \frac{\sigma}{h_N} \int_e [u_h] \cdot [v_h] \ ds
\end{align}
for any $v_h\in \bV$.

To preserve the accuracy of the original DG scheme,   interpolation operator $\mathcal{I}_h(\cdot)$ needs to reach certain accuracy. 
Using  similar   techniques as in \cite{DG4,huang2017quadrature}, we have the following proposition on local truncation error of the sparse grid method with $\bV=\hat{\bV}^k_N.$ We only discuss the case when $c(\bx)$ is discontinuous, since the similar approach can be applied when $c(\bx)$ is continuous.
\begin{prop}[Local truncation error analysis]
\label{prop:accurate-interp}
	If the interpolation operator $\mathcal{I}_h$ in \eqref{eq:DG-semi-interp-discontinuous} has the accuracy of order $\abs{\log_2h_N}^{d}h_N^{k+3}$ for sufficiently smooth functions, then the local truncation error of the semi-discrete DG scheme with interpolation \eqref{eq:DG-semi-interp-discontinuous} is of order $\abs{\log_2h_N}^{d}h_N^{k+1}$. To be more precise, for sufficiently smooth function $u$, the sparse grid DG method with interpolation \eqref{eq:DG-semi-interp-discontinuous} has the truncation error:
	\begin{equation}\label{eq:truncation-sparse}
		\norm{L_h(u) + \nabla \cdot (c^2(\mathbf{x}) \nabla u )}_{L^2(\Omega)}\le C \abs{\log_2h_N}^{d}h_N^{k+1}.
	\end{equation}
	Here, we use  $C$ to denote any generic constant that may depend on the solution $u$ and $c(\mathbf{x}),$ but does not depend on $N.$
\end{prop}
\begin{proof}
    We denote the standard $L^2$ projection operator onto the sparse grid DG finite element space by $\mathbf{P}$, then
	\begin{equation}
		\norm{L_h(u) + \nabla \cdot (c^2(\mathbf{x}) \nabla u)}\le e_1 + e_2,
	\end{equation}
	where 
	\begin{equation*}
		e_1 := \norm{L_h(u) + \mathbf{P}(\nabla \cdot (c^2(\mathbf{x}) \nabla u))},
	\end{equation*}
	and
	\begin{equation*}
		e_2 := \norm{\mathbf{P}(\nabla \cdot (c^2(\mathbf{x}) \nabla u) - \nabla \cdot (c^2(\mathbf{x}) \nabla u)}
	\end{equation*}
	The estimate for $e_2$ can be obtained by projection properties  \cite{guo2016transport}:
	\begin{equation}\label{eq:estimate-e2}
		e_2 \le C\abs{\log_2h_N}^{d}h_N^{k+1}.
	\end{equation}	

	To estimate $e_1,$ we consider  any test function $v_h$ in DG space, and obtain  
	\begin{align*}
		&\sum_{K\in {\Omega_N}}\int_{K}(L_h(u)+\mathbf{P}(\nabla \cdot (c^2(\mathbf{x}) \nabla u))v_h \ d\textbf{x}= \sum_{K\in {\Omega_N}}\int_{K}(L_h(u)+\nabla \cdot (c^2(\mathbf{x}) \nabla u))v_h \ d\textbf{x}\\
		={}& -  \int_{\Omega} \mathcal{I}_h(c^2 \nabla u) \cdot \nabla v_h \ d\textbf{x}  + \sum_{e\in \Gamma } \int_{e} \{\mathcal{I}_h(c^2 \nabla u) \} \cdot [v_h] \ ds \nonumber \\ 
		& + \sum_{e \in \Gamma}\int_{e} \brac{ \frac{1}{2}[\mathcal{I}_h((c^2)^- u)]\cdot(\nabla v_h)^- + \frac{1}{2}[\mathcal{I}_h((c^2)^+ u)]\cdot(\nabla v_h)^+} \ ds  - \sum_{e \in \Gamma} \frac{\sigma}{h_N} \int_e [u] \cdot [v_h] \ ds \\
		& +  \int_{\Omega} c^2 \nabla u \cdot \nabla v_h \ d\textbf{x}  - \sum_{e\in \Gamma } \int_{e} \{c^2 \nabla u\} \cdot [v_h] \ ds \nonumber \\ 
		& - \sum_{e \in \Gamma}\int_{e} \brac{ \frac{1}{2}[(c^2)^- u]\cdot(\nabla v_h)^- + \frac{1}{2}[(c^2)^+ u]\cdot(\nabla v_h)^+} \ ds  + \sum_{e \in \Gamma} \frac{\sigma}{h_N} \int_e [u] \cdot [v_h] \ ds \\
		={}& -  \int_{\Omega} \brac{\mathcal{I}_h(c^2 \nabla u)-c^2 \nabla u} \cdot \nabla v_h \ d\textbf{x}  
		+ \sum_{e\in \Gamma } \int_{e} \{ \mathcal{I}_h(c^2 \nabla u)-c^2 \nabla u \} \cdot [v_h] \ ds \nonumber \\ 
		& + \sum_{e \in \Gamma}\int_{e} \brac{ \frac{1}{2}[\mathcal{I}_h((c^2)^- u)-(c^2)^- u]\cdot(\nabla v_h)^- + \frac{1}{2}[\mathcal{I}_h((c^2)^+ u)-(c^2)^+ u]\cdot(\nabla v_h)^+} \ ds \\
		\le{}& \norm{\mathcal{I}_h[c^2 \nabla u]-c^2 \nabla u} \norm{\nabla v_h} + \norm{\mathcal{I}_h[c^2 \nabla u]-c^2 \nabla u}_{L^2(\Gamma_h)}\norm{v_h}_{L^2(\Gamma_h)}  \\
		& + \norm{\mathcal{I}_h[c^2 \nabla u]-c^2 \nabla u}_{L^2(\Gamma_h)}\norm{\nabla v_h}_{L^2(\Gamma_h)} \\
		\le{}& C\abs{\log_2h_N}^{d}h_N^{k+3}h_N^{-1}\norm{v_h}  + Ch_N^{-\frac12}\abs{\log_2h_N}^{d}h_N^{k+3}h_N^{-\frac{1}{2}}\norm{v_h}  + Ch_N^{-\frac12}\abs{\log_2h_N}^{d}h_N^{k+3}h_N^{-\frac32}\norm{v_h} \\
		={}& C\abs{\log_2h_N}^{d}h_N^{k+1}\norm{v_h}.
	\end{align*}
    Here, we have used the multiplicative trace inequality and the inverse inequality, see e.g. Lemma 2.1 and Lemma 2.3 in \cite{huang2017quadrature}.
	By taking $v_h=(L_h(u)+\mathbf{P}(\nabla \cdot (c^2(\mathbf{x}) \nabla u))$ in the  inequality above, we have
	\begin{equation}\label{eq:estimate-e1}
		e_1 = \norm{L_h(u)+\mathbf{P}(\nabla \cdot (c^2(\mathbf{x}) \nabla u))}\le C \abs{\log_2h_N}^{d}h_N^{k+1}.
	\end{equation}	
    Combining \eqref{eq:estimate-e1} and \eqref{eq:estimate-e2}, we have the estimate for the truncation error \eqref{eq:truncation-sparse}.
\end{proof}

\begin{rem}
 	The proposition above 
	indicates that, to  preserve the order of the original scheme, we should use $M \ge k+2.$ For example, if we take piecewise linear polynomials for the DG space, then it is required to apply cubic interpolation operator to treat the nonlinear terms. From our numerical tests, however, this seems that it is not a necessary condition. To reach the desired convergence rate, one only needs to take $M\ge k$.
\end{rem}
    
In Proposition \ref{prop:accurate-interp}, we only estimate the local truncation error, and this is far from a rigorous error estimate that takes into account  stability. Unlike the scheme with the symmetric bilinear form $B(u_h,v)$ as in Theorem \ref{thm:stable}, the symmetry is lost in the interpolated bilinear form $\tilde{B}(u_h, v).$  Hence, energy stability is not automatic. In numerical experiments, we observe that the sparse grid DG method with Lagrange interpolation with only inner interface points is unstable for polynomials of high degrees (see the numerical results in Table \ref{tab:smooth-sparse-2D-inner-pt} in Section 4). With the interpolation points at the {\emph{interface}}, the sparse grid DG scheme is   stable and yields satisfactory convergence rate (see Table \ref{tab:smooth-sparse-2D-interface-pt} in Section 4). 

\subsection{Time stepping and adaptivity}
\label{subsec:timeadapt}

For   time discretizations, we first write the second order semi-discrete scheme \eqref{eq:DG-semi-interp-discontinuous}
\begin{equation}
	(u_h)_{tt} = L_h(u_h)
\end{equation}
into a first order system
\begin{align*}
	(u_h)_t &= w_h, \\
	(w_h)_t &= L_h(u_h),
\end{align*}
and then appy the standard Runge-Kutta scheme. The reason why we use the one-step RK method instead of the multistep method is that the maximum allowed time step size from the CFL restriction may change with the adaptive mesh in different time steps. This would result in additional computational cost in extrapolation or interpolation between different time steps for the multistep methods. 

The adaptive scheme uses the procedure developed in \cite{bokanowski2013adaptive,guo2017adaptive} to determine the space $\bV$ that dynamically evolves over time. The method is very similar to those in \cite{bokanowski2013adaptive,guo2017adaptive}, and the details are omitted for brevity. The main difference is that we need to keep track of  two sets of basis functions corresponding to the same adaptive space are involved  \cite{huang2019adaptive}.  Another difference is that the refinement and the coarsening criteria are determined by the $L^2$ norms of both $u_h$ and $w_h$, which are both important  for predicting solution profiles for wave equations. There are some cases which start with a zero displacement $u$ but a non-zero velocity $u_t$. If we only take the norms of $u_h$ as an indicator, the adaptive procedure will result in poor resolutions. Only by considering the norms of both $u_h$ and $w_h$, one can capture the profiles well.

\subsection{Fast algorithms}
\label{subsec:fast}

We now describe the fast matrix-vector multiplication algorithm, which is essential for efficient implementation of our schemes. Because the multiwavelet bases are global,  the evaluation of the residual yields denser matrix than those obtained by standard local bases. Efficient implementations are therefore essential to ensure that the computational cost is on par with element-wise implementation of traditional DG schemes. This issue has been also discussed in our work for conservation laws \cite{huang2019adaptive}, which extends the fast matrix-vector multiplication in \cite{shen2010efficient,zeiser2011fast} to adaptive index set.

Following \cite{shen2010efficient,huang2019adaptive}, we consider   matrix-vector multiplication in multi-dimensions in an abstract framework.
 \begin{equation}\label{eq:LU-original}
	f_{\bm{n}} = \sum_{H(\bm{n}')\le 0} f'_{\bm{n}'} t_{n_1', n_1}^{(1)} t_{n_2',n_2}^{(2)}\cdots t_{n_d',n_d}^{(d)}, \quad H(\bm{n})\le 0,
\end{equation}
where $\bm{n}=({n}_1,{n}_2,\dots,{n}_d)$ and $\bm{n}'=({n}_1',{n}_2',\dots,{n}_d')$ can be thought of as  the level of the mesh, and $t_{n_1', n_1}^{(i)}=T^{(i)}_{n_1', n_1}$ represents the calculations in the $i$-th dimension. It is assumed that the constraint function $H=H(\bm{n}')=H({n}_1',{n}_2',\dots,{n}_d')$ is non-decreasing with respect to each variable. This holds true for sparse grid (by taking $H(\bm{n}')=|\bm{n}'|_1$) and also for adaptive multiresolution method.

One can compute the sum \eqref{eq:LU-original} dimension-by-dimension, i.e. we first perform the transformation in the $x_1$ dimension:
\begin{equation}\label{eq:adapt-x1}
    g^{(1)}_{(n_1,n_2',\dots,n_d')} = \sum_{H(n_1',n_2',\dots,n_d')\le 0} f'_{(n_1',n_2',\dots,n_d')} t_{n_1', n_1}^{(1)},
\end{equation}
and then in the $x_2$ dimension:
\begin{equation}\label{eq:adapt-x2}
    g^{(2)}_{(n_1,n_2,n_3'\dots,n_d')} = \sum_{H(n_1,n_2',\dots,n_d')\le 0} g^{(1)}_{(n_1,n_2',\dots,n_d')} t_{n_2', n_2}^{(2)},
\end{equation}
and all the way up to $x_d$ dimension:
\begin{equation}\label{eq:adapt-xd}
    f_{(n_1,n_2,n_3\dots,n_d)} = \sum_{H(n_1,n_2,\dots,n_{d-1},n_d')\le 0} g^{(d-1)}_{(n_1,n_2,\dots,n_{d-1},n_d')} t_{n_d', n_d}^{(d)}.
\end{equation}
It can be proved that    \eqref{eq:adapt-x1}-\eqref{eq:adapt-xd} is equivalent to the original summation \eqref{eq:LU-original}, if assuming that, for some integer $1\le k\le d$, $T^{(i)}$ for $i=1,\dots,k-1$ are strictly lower triangular and $T^{(i)}$ for $i=k+1,\dots,d$ are upper triangular (or $T^{(i)}$ for $i=1,\dots,k-1$ are lower triangular and $T^{(i)}$ for $i=k+1,\dots,d$ are strictly upper triangular) \cite{shen2010efficient}. Here, $T^{(i)}$ denotes the $i$-th transformation matrix.
When such properties for $T^{(i)}$ matrices are not true, one can perform  $L+U$ split and   \eqref{eq:LU-original} becomes:
\begin{equation}
    f_{\bm{n}} = \sum_{H(\bm{n}')\le 0} f'_{\bm{n}'} (l_{n_1', n_1}^{(1)}+u_{n_1', n_1}^{(1)})(l_{n_1', n_1}^{(1)}+u_{n_1', n_1}^{(1)})\cdots (l_{n_{d-1}', n_{d-1}}^{(d-1)}+u_{n_{d-1}', n_{d-1}}^{(d-1)})t_{n_d',n_d}^{(d)},
\end{equation}
where there are in total $2^{d-1}$ terms that can be computed dimension-by-dimension. 
The overall computational cost   is $\mathcal{O}(2^{d-1} \cdot DoF \cdot N)$   if the cost of one-dimensional transform is log-linear, i.e., $\mathcal{O}(\mathcal{N}\log \mathcal{N})$  where $\mathcal{N}$ denotes the DoF in one-dimension \cite{shen2010efficient}. This assumption holds true for our sparse grid DG scheme.  

We apply this fast matrix-vector multiplication in several parts of our algorithm. We will discuss the details about   initialization, which is the procedure to project the initial value onto the DG finite element space represented by multiwavelet bases. When the given initial value is separable, i.e., 
\begin{equation}
	u(x) = \prod_{i=1}^d u_i(x_i),
\end{equation}
one just need to project each 1D function $u_i=u_i(x_i)$ for $i=1,\dots,d$ onto 1D multiwavelet bases and then we can easily get the projection of $u(x)$. This approach naturally extends to the case when the initial value is a summation of separable functions:
\begin{equation}
	u(x) = \sum_{j=1}^n (\prod_{i=1}^d u_{j,i}(x_i)).
\end{equation}
However, if the function is non-separable, direct evaluation of $L^2$ projection would result in very large computational cost if using numerical quadratures in multi-dimensions. Rather, we propose to apply the  adaptive multiresolution interpolation introduced in \cite{tao2019collocation} with appropriate error tolerance and approximate the solution using a collection of interpolation basis functions:
\begin{equation}\label{eq:init-interp-u}
	u_h(\bx) = \sum_{\substack{(\bl,\bj)\in G, \\ \mathbf{1}\le \bi\le \bk+\mathbf{1} }} {b}^\bj_{\bi,\bl} \psi^\bj_{\bi,\bl}(\bx)
\end{equation}
with $G$ the index of all active elements. Next we use the fast matrix-vector multiplication \eqref{eq:LU-original} to transform coefficients of interpolation basis $\{ {b}^\bj_{\bi,\bl} \}$ to coefficients of Alpert's basis $\{ {c}^\bj_{\bi,\bl} \}$:
\begin{equation}\label{eq:init-alpert-u}
	u_h(\bx) = \sum_{\substack{(\bl,\bj)\in G, \\ \mathbf{1}\le \bi\le \bk+\mathbf{1} }} {c}^\bj_{\bi,\bl} v^\bj_{\bi,\bl}(\bx)
\end{equation}
In \eqref{eq:LU-original}, $f_{\bm{n}}$ and $f'_{\bm{n}'}$ represent $\{ {b}^\bj_{\bi,\bl} \}$ and $\{ {c}^\bj_{\bi,\bl} \}$, respectively. The matrix $t_{n_i',n_i}^{(i)}$ is the product of 1D interpolation basis and 1D Alpert's basis. 

We apply similar approach in the multiresolution interpolation and the evaluation of the right hand side of the weak formulation  \eqref{eq:DG-semi-interp-discontinuous}. We refer readers to \cite{huang2019adaptive} for details.

\section{Numerical examples}
\label{sec:numeric}

In this section, we perform numerical experiments to validate the performance of our scheme. We consider 2D or 3D problems with computational domain being $[0,1]^d$ with $d=2,3$. The CFL number is taken to be $0.1$ in 2D and $0.05$ in 3D. The penalty parameter $\sigma$ is taken to be 10 in 2D and 30 in 3D, unless otherwise stated. For the accuracy test, we check the convergence order for $P^k$ DG with $k=1,2,3$ coupled with RK time discretization. In particular, for $k=1,2,$ we use the second and the third-order strong stability preserving Runge-Kutta method \cite{shu1988jcp,Gottlieb_2001_SIAM_Stabi}, and for $k=3,$ we use the classical RK4 methods. All adaptive calculations are obtained by $k=3$ and RK4 time stepping. In the adaptive scheme, we take $\eta=\epsilon/10$.  $\textrm{DoF=dim}(\bV^k)$ refers to the number of Alperts' multiwavelets basis functions in the adaptive grids. The maximum mesh level $N$ is taken to be 8, unless otherwise stated.

\begin{exam}[wave equation with constant coefficient]\label{exam:01-const-coeff}
In this example, consider the $d$-dimensional wave equation with a constant coefficient
\begin{equation}
    u_{tt} =  \sum_{i=1}^du_{x_ix_i} \\
\end{equation}
on the domain $[0, 1]^d$. We take the exact solution to be 
$$u(\mathbf{x},t)= \sin(a\sqrt{d}\pi t)\prod\limits_{i=1}^{d} \cos(a\pi x_i)$$
with $a$ being a constant and various types of boundary conditions. 

\begin{enumerate}[label=(\alph*)]
\item We take $a=2$ and $d=2, 3$ with periodic boundary conditions.

\item We take $a=1$ and $d=2, 3$ and incorporate Dirichlet boundary condition in the $x_1$-direction and Neumann boundary in other directions.
\end{enumerate}
\end{exam}

Note that in this example, \eqref{sp2} is implemented with no interpolation because $c$ is a constant. 
 To output the $L^2$-error between the numerical solution $u_h$ and the exact solution $u(x),$
we use the fact that \begin{align*}
 \int_{\Omega} (u_h - u)^2 dx  = \int_{\Omega} u_h^2 dx - 2\int_{\Omega} u_hu dx + \int_{\Omega} u^2 dx.
\end{align*}
The first term $\int_{\Omega} u_h^2 dx$ can be easily computed with the aid of the orthonormality of the Alpert's basis functions. The second term $\int_{\Omega} u_hu dx$ can be computed by the same fast approach as the initial projection, which has been explained in detail in Section \ref{subsec:fast}. The third term $\int_{\Omega} u^2 dx$ can be computed analytically.   

 The numerical results obtained by sparse grid DG method are presented in Table \ref{tab:constant-periodic-sparse} for case (a) and in Table \ref{tab:constant-mixbc-sparse} for case (b). For both cases, the convergence order is slightly bigger than $k$ but smaller than $k+1$, which is higher than the predicted rate in Theorem \ref{thm:error}, but similar to the results for linear transport equation in \cite{guo2016transport}. The numerical results with adaptive method are shown in Tables \ref{tab:constant-periodic-adaptive} and   \ref{tab:constant-mixbc-adaptive}. Similar to  \cite{guo2017adaptive}, we measure the convergence rates with respect to DoF: $R_{\textrm{DoF}}$ and $\epsilon: R_{\epsilon}.$ We can clearly observe the effectiveness of the adaptive algorithm, i.e. $R_{\epsilon}$ is close to 1. The convergence order $R_{\textrm{DoF}}$ is bigger than $\frac{k+1}{d},$ which is the rate obtained by an optimally convergent non-adaptive scheme.

\begin{table}[htbp]
  \centering
  \caption{Example \ref{exam:01-const-coeff}(a): wave equation with constant coefficients, periodic boundary conditions, sparse grid DG, $k=1, 2, 3$, $d=2, 3$. $t=0.1$.}
  \label{tab:constant-periodic-sparse}%
    \begin{tabular}{c|c|c|c|c|c|c|c|c|c}
    \hline
    \multirow{2}[4]{*}{} & \multirow{2}[4]{*}{$N$} & \multicolumn{2}{c|}{$k=1$} & \multirow{2}[4]{*}{$N$} & \multicolumn{2}{c|}{$k=2$} & \multirow{2}[4]{*}{$N$} & \multicolumn{2}{c}{$k=3$} \bigstrut\\
\cline{3-4}\cline{6-7}\cline{9-10}          &      &   $L^2$-error    &    order   &       &   $L^2$-error    &    order   &       &   $L^2$-error    & order \bigstrut\\
    \hline
    \multirow{4}[2]{*}{$d=2$}     	  
          & 5 & 5.90e-03 & -    & 5  & 1.96e-04 & -     & 3  & 2.80e-04 & - \\
          & 6 & 1.69e-03 & 1.81 & 6  & 3.03e-05 & 2.69  & 4  & 1.80e-05 & 3.96 \\
          & 7 & 4.66e-04 & 1.86 & 7  & 4.43e-06 & 2.77  & 5  & 1.48e-06 & 3.61 \\
          & 8 & 1.23e-04 & 1.92 & 8  & 6.21e-07 & 2.83  & 6  & 1.10e-07 & 3.76 \\
    \hline
    \hline
    \multirow{2}[4]{*}{} & \multirow{2}[4]{*}{$N$} & \multicolumn{2}{c|}{$k=1$} & \multirow{2}[4]{*}{$N$} & \multicolumn{2}{c|}{$k=2$} & \multirow{2}[4]{*}{$N$} & \multicolumn{2}{c}{$k=3$} \bigstrut\\
\cline{3-4}\cline{6-7}\cline{9-10}          &      &   $L^2$-error    &    order   &       &   $L^2$-error    &    order   &       &   $L^2$-error    & order \bigstrut\\
    \hline
    \multirow{4}[2]{*}{$d=3$} 
			& 5 & 1.58e-02 & -    & 5 & 7.38e-04 & -    & 3 & 4.50e-04 & -    \\
			& 6 & 8.66e-03 & 0.87 & 6 & 1.68e-04 & 2.14 & 4 & 8.02e-05 & 2.49 \\
			& 7 & 2.42e-03 & 1.84 & 7 & 3.03e-05 & 2.47 & 5 & 5.24e-06 & 3.94 \\
			& 8 & 8.41e-04 & 1.53 & 8 & 5.30e-06 & 2.51 & 6 & 5.30e-07 & 3.30 \\
	\hline
    \end{tabular}%
\end{table}%

\begin{table}[htbp]
  \centering
  \caption{Example \ref{exam:01-const-coeff}(b): wave equation with constant coefficients, Dirichlet and Neumann boundary conditions, sparse grid DG, $k=1, 2, 3$, $d=2, 3$. $t=0.1$.}
  \label{tab:constant-mixbc-sparse}%
    \begin{tabular}{c|c|c|c|c|c|c|c|c|c}
    \hline
    \multirow{2}[4]{*}{} & \multirow{2}[4]{*}{$N$} & \multicolumn{2}{c|}{$k=1$} & \multirow{2}[4]{*}{$N$} & \multicolumn{2}{c|}{$k=2$} & \multirow{2}[4]{*}{$N$} & \multicolumn{2}{c}{$k=3$} \bigstrut\\
\cline{3-4}\cline{6-7}\cline{9-10}          &      &   $L^2$-error    &    order   &       &   $L^2$-error    &    order   &       &   $L^2$-error    & order \bigstrut\\
    \hline
    \multirow{4}[2]{*}{$d=2$}     	  
          & 3 & 3.30e-03 & -    & 3 & 1.21e-04 & -     & 1 & 4.51e-04 & - \\
          & 4 & 1.21e-03 & 1.44 & 4 & 1.79e-05 & 2.75  & 2 & 6.34e-05 & 2.83 \\
          & 5 & 2.94e-04 & 2.04 & 5 & 2.43e-06 & 2.89  & 3 & 8.30e-06 & 2.93 \\
          & 6 & 8.15e-05 & 1.85 & 6 & 3.41e-07 & 2.83  & 4 & 7.64e-07 & 3.44 \\
    \hline
    \hline
    \multirow{2}[4]{*}{} & \multirow{2}[4]{*}{$N$} & \multicolumn{2}{c|}{$k=1$} & \multirow{2}[4]{*}{$N$} & \multicolumn{2}{c|}{$k=2$} & \multirow{2}[4]{*}{$N$} & \multicolumn{2}{c}{$k=3$} \bigstrut\\
\cline{3-4}\cline{6-7}\cline{9-10}          &      &   $L^2$-error    &    order   &       &   $L^2$-error    &    order   &       &   $L^2$-error    & order \bigstrut\\
    \hline
    \multirow{4}[2]{*}{$d=3$} 
			& 3 & 2.15e-02 & -    & 3 & 2.14e-04 & - & 1 & 4.44e-04 & -  \\
			& 4 & 7.06e-03 & 1.61 & 4 & 3.39e-05 & 2.66 & 2 & 3.66e-05 & 3.60 \\
			& 5 & 2.04e-03 & 1.79 & 5 & 5.13e-06 & 2.73 & 3 & 2.39e-06 & 3.94 \\
			& 6 & 5.27e-04 & 1.95 & 6 & 1.08e-06 & 2.24 & 4 & 1.23e-07 & 4.28 \\
	\hline
    \end{tabular}%
\end{table}%

\begin{table}[!hbp]
\centering
\caption{Example \ref{exam:01-const-coeff}(a): wave equation with constant coefficients, periodic boundary conditions, adaptive sparse grid DG, $k=3$.}
\label{tab:constant-periodic-adaptive}
\begin{tabular}{c|c|c|c|c|c}
  \hline
  & $\epsilon$ & DoF & $L^2$-error & $R_{\textrm{DoF}}$ & $R_{\epsilon}$ \\
  \hline
\multirow{4}{3em}{$d=2$}
& 1e-1 & 128 & 1.25e-3 & - & - \\
& 1e-2 & 320 & 2.80e-4 & 1.63 & 0.65 \\
& 1e-3 & 1088 & 2.61e-5 & 1.94 & 1.03 \\
& 1e-4 & 1536 & 4.02e-6 & 5.43 & 0.81 \\
\hline
\multirow{4}{3em}{$d=3$}
& 1e-1 & 896 & 2.16e-3 & - & - \\
& 1e-2 & 2432 & 4.74e-4 & 1.52 & 0.66 \\
& 1e-3 & 5888 & 8.30e-5 & 1.97 & 0.76  \\
& 1e-4 & 28160 & 6.69e-6 & 1.61 & 1.09 \\   
\hline
\end{tabular}
\end{table}

\begin{table}[!hbp]
\centering
\caption{Example \ref{exam:01-const-coeff}(b): wave equation with constant coefficients, Dirichlet and Neumann boundary conditions, adaptive sparse grid DG, $k=3$.}
\label{tab:constant-mixbc-adaptive}
\begin{tabular}{c|c|c|c|c|c}
  \hline
  & $\epsilon$ & DoF & $L^2$-error & $R_{\textrm{DoF}}$ & $R_{\epsilon}$ \\
  \hline
\multirow{4}{3em}{$d=2$}
& 1e-1 & 32 & 1.11e-3 & - & - \\
& 1e-2 & 112 & 6.35e-5 & 2.29 & 1.24 \\
& 1e-3 & 208 & 1.70e-5 & 2.13 & 0.57 \\
& 1e-4 & 384 & 2.35e-6 & 3.22 & 0.86 \\   
\hline
\multirow{4}{3em}{$d=3$}
& 1e-1 & 64 & 1.29e-3 & - & - \\
& 1e-2 & 640 & 3.71e-5 & 1.54 & 1.54 \\
& 1e-3 & 1280 & 2.41e-5 & 0.62 & 0.19 \\
& 1e-4 & 2368 & 3.12e-6 & 3.32 & 0.89 \\
\hline
\end{tabular}
\end{table}

\begin{exam}[wave equation with smooth variable coefficient]\label{exam:02-smooth-var-coeff}
This example tests wave equation with smooth variable coefficient
\begin{equation}
	u_{tt} - \nabla \cdot (c^2(\mathbf{x}) \nabla u )= f,
\end{equation}
on the computational domain $[0, 1]^d$ with $d=2,3$ and periodic boundary conditions.

For 2D case, we take
\begin{equation}
  c^2(x_1,x_2)= (\cos(2\pi x_1)\cos(2\pi x_2) +2)/3,
\end{equation}
and the corresponding source term $f=f(x_1,x_2,t)$ such that the exact solution is
\begin{equation}
  u= \sin(\pi t) \sin(2\pi x_1)\cos(2\pi x_2).
\end{equation}

For 3D case, we take 
\begin{equation}
  c^2(x_1,x_2,x_3) = (\sin(2\pi x_1)\sin(2\pi x_2)\cos(2\pi x_3) +2)/3,
\end{equation}
and the corresponding source term $f=f(x_1,x_2,x_3,t)$ such that the exact solution is
\begin{equation}
  u = \sin(\pi t) \sin(2\pi x_1)\cos(2\pi x_2)\cos(2\pi x_3).
\end{equation}
\end{exam}

This problem needs to invoke the fast interpolation methods to handle the variable coefficient. We first compare different choices of interpolation points. 
We use the inner interpolation points in Table \ref{tab:smooth-sparse-2D-inner-pt} for 2D.  The interpolation points and the basis functions are listed in the Appendix. When $k=1$, the convergence order seems satisfactory. However, for $k=2, 3,$ the scheme is unstable. If we use Lagrange interpolation with the interface points, one will observe good convergence rate for $M\ge k$, as shown in Table \ref{tab:smooth-sparse-2D-interface-pt}.  We also find that the error is almost the same for $M=k+1$ and $M=k+2$, and both much smaller than $M=k$. Therefore, in applications, we recommend taking $M=k+1$ for accuracy considerations. Notice that this is a more relaxed condition from what is indicated by the local truncation analysis Proposition \ref{prop:accurate-interp}. We also remark that for nonlinear conservation laws in \cite{huang2019adaptive},   Lagrange interpolation is unstable even with interface points, and Hermite interpolation has to be employed. However, for all numerical examples in this paper  for linear wave equations with variable coefficients, Lagrange interpolation with interface points yields a stable scheme, and we choose to use this instead of Hermite interpolation due to its easier implementation. 

For 3D cases, to save space, we only show numerical results with interface interpolation points in Table \ref{tab:smooth-sparse-3D}, in which good convergence rate is also observed. The result using adaptive method with $k=3$ and $M=4$ are presented in Table \ref{tab:smooth-adaptive} for both 2D and 3D, and the conclusions are similar to the constant coefficient case.

\begin{table}[htbp]
  \centering
  \caption{Example \ref{exam:02-smooth-var-coeff}: wave equation with smooth variable coefficients in 2D, sparse grid DG, Lagrange interpolation with inner interpolation points, $k=1, 2, 3$. $t=0.1$.}
  \label{tab:smooth-sparse-2D-inner-pt}%
    \begin{tabular}{c|c|c|c|c|c|c|c}
    \hline
    \multirow{2}[4]{*}{} & \multirow{2}[4]{*}{$N$} & \multicolumn{2}{c|}{$M=1$} & \multicolumn{2}{c|}{$M=2$} & \multicolumn{2}{c}{$M=3$} \bigstrut\\
    \cline{3-8}          &       & $L^2$-error & order & $L^2$-error & order & $L^2$-error & order \bigstrut\\
    \hline
    \multirow{4}[3]{3em}{$k=1$} & 3     & 2.52e-02 & - & 2.55e-02 & - & 2.52e-02 & - \\
                              & 4     & 1.68e-02 & 0.59 & 1.63e-02 & 0.64 & 1.63e-02 & 0.63 \\
                              & 5     & 3.67e-03 & 2.19 & 3.37e-03 & 2.28 & 3.36e-03 & 2.28 \\
                              & 6     & 9.43e-04 & 1.96 & 9.62e-04 & 1.81 & 8.33e-04 & 2.01 \\
    \hline
    \hline
    \multirow{2}[4]{*}{} & \multirow{2}[4]{*}{$N$} & \multicolumn{2}{c|}{$M=2$} & \multicolumn{2}{c|}{$M=3$} & \multicolumn{2}{c}{$M=4$} \bigstrut\\
    \cline{3-8}          &       & $L^2$-error & order & $L^2$-error & order & $L^2$-error & order \bigstrut\\
    \hline
    \multirow{4}[3]{3em}{$k=2$}
                              & 3     & 1.64e-02 & - & 3.95e-03 & -  & 4.73e-02 & -  \\
                              & 4     & 1.03e-02 & 0.67 & 7.89e-04 & 2.32  & 1.46e-01 & -1.63  \\
                              & 5     & 3.57e-03 & 1.53 & 7.80e-04 & 0.02  & 3.17e+02 & -11.08 \\
                              & 6     & 1.13e-02 & -1.66& 1.60e-01 & -7.68 & 6.27e+10 & -27.56 \\
    \hline  
    \hline  
    \multirow{2}[4]{*}{} & \multirow{2}[4]{*}{$N$} & \multicolumn{2}{c|}{$M=3$} & \multicolumn{2}{c|}{$M=4$} & \multicolumn{2}{c}{$M=5$} \bigstrut\\
    \cline{3-8}          &       & $L^2$-error & order & $L^2$-error & order & $L^2$-error & order \bigstrut\\
    \hline
    \multirow{4}[3]{3em}{$k=3$} 
                              & 3 & 5.28e-03 & - & 2.14e+03 & - & 5.13e+03 & - \\
                              & 4 & 5.55e-02 & -3.40   & 1.11e+10 & -22.31   & 2.23e+09 & -18.73 \\
                              & 5 & 1.93e+03 & -15.09  & 1.39e+24 & -46.83   & 4.78e+21 & -40.97 \\
                              & 6 & 8.15e+20 & -58.55  & 9.83e+56 & -109.12  & 1.40e+51 & -97.88 \\
    \hline    
    \end{tabular}%
\end{table}%

\begin{table}[htbp]
  \centering
  \caption{Example \ref{exam:02-smooth-var-coeff}: wave equation with smooth variable coefficients in 2D, sparse grid DG, Lagrange interpolation with interface interpolation points, $k=1, 2, 3$. $t=0.1$.}
  \label{tab:smooth-sparse-2D-interface-pt}%
    \begin{tabular}{c|c|c|c|c|c|c|c}
    \hline
    \multirow{2}[4]{*}{} & \multirow{2}[4]{*}{$N$} & \multicolumn{2}{c|}{$M=2$} & \multicolumn{2}{c|}{$M=3$} & \multicolumn{2}{c}{$M=4$} \bigstrut\\
    \cline{3-8}          &       & $L^2$-error & order & $L^2$-error & order & $L^2$-error & order \bigstrut\\
    \hline
    \multirow{4}[3]{3em}{$k=1$}
                              & 3 & 2.52e-02 & - & 2.52e-02 & - & 2.52e-02 & - \\
                              & 4 & 1.65e-02 & 0.61 & 1.63e-02 & 0.63 & 1.63e-02 & 0.63 \\
                              & 5 & 3.52e-03 & 2.23 & 3.36e-03 & 2.28 & 3.36e-03 & 2.28 \\
                              & 6 & 9.52e-04 & 1.89 & 8.30e-04 & 2.02 & 8.27e-04 & 2.02 \\
    \hline
    \hline
    \multirow{2}[4]{*}{} & \multirow{2}[4]{*}{$N$} & \multicolumn{2}{c|}{$M=2$} & \multicolumn{2}{c|}{$M=3$} & \multicolumn{2}{c}{$M=4$} \bigstrut\\
    \cline{3-8}          &       & $L^2$-error & order & $L^2$-error & order & $L^2$-error & order \bigstrut\\
    \hline
    \multirow{4}[3]{3em}{$k=2$} 
                              & 3 & 2.69e-03 & - & 2.08e-03 & - & 2.08e-03 & - \\
                              & 4 & 5.24e-04 & 2.36 & 4.38e-04 & 2.25 & 4.37e-04 & 2.25 \\
                              & 5 & 1.25e-04 & 2.07 & 7.58e-05 & 2.53 & 7.58e-05 & 2.53 \\
                              & 6 & 1.64e-05 & 2.93 & 1.16e-05 & 2.71 & 1.16e-05 & 2.71 \\
    \hline
    \hline
   \multirow{2}[4]{*}{} & \multirow{2}[4]{*}{$N$} & \multicolumn{2}{c|}{$M=3$} & \multicolumn{2}{c|}{$M=4$} & \multicolumn{2}{c}{$M=5$} \bigstrut\\
    \cline{3-8}          &       & $L^2$-error & order & $L^2$-error & order & $L^2$-error & order \bigstrut\\
    \hline
    \multirow{4}[3]{3em}{$k=3$}
                              & 3 & 2.92e-04 & - & 9.28e-05 & - & 8.75e-05 & - \\
                              & 4 & 2.66e-05 & 3.46 & 1.05e-05 & 3.15 & 1.03e-05 & 3.09 \\
                              & 5 & 3.04e-06 & 3.13 & 7.80e-07 & 3.74 & 7.68e-07 & 3.74 \\
                              & 6 & 1.83e-07 & 4.05 & 5.10e-08 & 3.94 & 5.03e-08 & 3.93 \\
    \hline    
    \end{tabular}%
\end{table}%

\begin{table}[htbp]
  \centering
  \caption{Example \ref{exam:02-smooth-var-coeff}: wave equation with smooth variable coefficients in 3D, sparse grid DG, Lagrange interpolation with interface interpolation points, $k=1, 2, 3$. $t=0.1$.}
  \label{tab:smooth-sparse-3D}%
    \begin{tabular}{c|c|c|c|c|c|c|c}
    \hline
    \multirow{2}[4]{*}{} & \multirow{2}[4]{*}{$N$} & \multicolumn{2}{c|}{$M=1$} & \multicolumn{2}{c|}{$M=2$} & \multicolumn{2}{c}{$M=3$} \bigstrut\\
    \cline{3-8}          &       & $L^2$-error & order & $L^2$-error & order & $L^2$-error & order \bigstrut\\
    \hline
    \multirow{4}[3]{3em}{$k=1$} 
                              & 3 & 1.17e-01 & - & 1.17e-01 & - & 1.17e-01 & - \\
                              & 4 & 2.20e-02 & 2.41  & 2.20e-02 & 2.41  & 2.20e-02 & 2.41  \\
                              & 5 & 1.74e-02 & 0.34  & 1.71e-02 & 0.36  & 1.71e-02 & 0.36  \\
                              & 6 & 4.65e-03 & 1.90  & 4.52e-03 & 1.92  & 4.51e-03 & 1.92  \\
    \hline
    \hline
    \multirow{2}[4]{*}{} & \multirow{2}[4]{*}{$N$} & \multicolumn{2}{c|}{$M=2$} & \multicolumn{2}{c|}{$M=3$} & \multicolumn{2}{c}{$M=4$} \bigstrut\\
    \cline{3-8}          &       & $L^2$-error & order & $L^2$-error & order & $L^2$-error & order \bigstrut\\
    \hline
    \multirow{4}[3]{3em}{$k=2$}
                              & 4 & 2.96e-03 & -    & 1.58e-03 & -    & 1.58e-03 & -   \\
                              & 5 & 7.78e-04 & 1.93 & 3.28e-04 & 2.27 & 3.27e-04 & 2.27 \\
                              & 6 & 2.93e-04 & 1.41 & 6.58e-05 & 2.32 & 6.58e-05 & 2.32 \\
                              & 7 & 3.88e-05 & 2.92 & 1.15e-05 & 2.52 & 1.15e-05 & 2.52 \\
    \hline
    \hline  
    \multirow{2}[4]{*}{} & \multirow{2}[4]{*}{$N$} & \multicolumn{2}{c|}{$M=3$} & \multicolumn{2}{c|}{$M=4$} & \multicolumn{2}{c}{$M=5$} \bigstrut\\
    \cline{3-8}          &       & $L^2$-error & order & $L^2$-error & order & $L^2$-error & order \bigstrut\\
    \hline
    \multirow{4}[3]{3em}{$k=3$} 
                              & 3 & 8.96e-04 & - & 3.88e-04 & - & 3.17e-04 & - \\
                              & 4 & 2.05e-04 & 2.13 & 3.58e-05 & 3.44 & 2.19e-05 & 3.85 \\
                              & 5 & 4.87e-05 & 2.07 & 3.27e-06 & 3.45 & 3.01e-06 & 2.86 \\
                              & 6 & 5.80e-06 & 3.07 & 2.55e-07 & 3.68 & 2.31e-07 & 3.71 \\

    \hline    
    \end{tabular}%
\end{table}%

\begin{table}[!hbp]
\centering
\caption{Example \ref{exam:02-smooth-var-coeff}, wave equation with smooth variable coefficient, adaptive sparse grid DG, 2D and 3D. $k=3$, $M=4$, $t=0.1$.}
\label{tab:smooth-adaptive}
\begin{tabular}{c|c|c|c|c|c}
  \hline
  & $\epsilon$ & DoF & $L^2$-error & $R_{\textrm{DoF}}$ & $R_{\epsilon}$ \\
  \hline
\multirow{4}{3em}{$d=2$}
& 1e-1 & 96 & 1.66e-3 & - & - \\
& 1e-2 & 224 & 3.03e-4 & 2.00 & 0.74 \\
& 1e-3 & 672 & 2.78e-5 & 2.18 & 1.04 \\
& 1e-4 & 1088 & 3.17e-6 & 4.50 & 0.94 \\   
\hline
\multirow{4}{3em}{$d=3$}
& 1e-1 & 576 & 2.11e-3 & - & - \\
& 1e-2 & 1152 & 5.26e-4 & 2.00 & 0.60 \\
& 1e-3 & 3584 & 8.73e-5 & 1.58 & 0.78 \\
& 1e-4 & 8704 & 1.26e-5 & 2.18 & 0.84 \\
\hline
\end{tabular}
\end{table}

\begin{exam}[wave equation with discontinuous coefficients]\label{exam:03-dis-coeff}
In this example, we consider  wave equation with discontinuous coefficients. The jump of the coefficient aligns with the cell interface on the fine mesh $\Omega_N.$

For 2D case, the domain $\Omega = [0, 1]^2$ is composed of two subdomains $\Omega_1 = [\frac{1}{4}, \frac{3}{4}] \times [0, 1]$ and $\Omega_2 = \Omega \backslash \Omega_1$. The coefficient $c^2$ is  a constant in each subdomain:
\begin{align}\label{example3_2_new}
c^2=
\begin{cases}
1, \qquad \text{in} \quad \Omega_1, \\
{\frac{5}{37}}, \qquad \text{in} \quad \Omega_2.
\end{cases}
\end{align}
Periodic boundary conditions are imposed in both $x_1$- and $x_2$- directions. With this setup, the exact solution is a standing wave
\begin{align}\label{example3_3_new}
u= 
\begin{cases}
\sin(\sqrt{20}\pi t)\cos(4\pi x_1) \cos(2\pi x_2 ), \qquad \text{in} \quad \Omega_1, \\
\sin(\sqrt{20}\pi t)\cos(12\pi x_1 )\cos(2\pi x_2), \qquad \text{in} \quad \Omega_2.
\end{cases}
\end{align}

For 3D case, $\Omega_1 = [\frac{1}{4}, \frac{3}{4}] \times [0, 1] \times [0, 1]$ and $\Omega_2 = \Omega \backslash \Omega_1$
\begin{align}\label{example3_2_3d}
c^2=
\begin{cases}
1, \qquad \text{in} \quad \Omega_1, \\
{\frac{3}{19}}, \qquad \text{in} \quad \Omega_2.
\end{cases}
\end{align}
Periodic boundary conditions are imposed in all   directions. With this setup, the exact solution is a standing wave
\begin{align}\label{example3_3_3d}
u= 
\begin{cases}
\sin(\sqrt{24}\pi t)\cos(4\pi x_1) \cos(2\pi x_2 ) \cos(2\pi x_3), \qquad \text{in} \quad \Omega_1, \\
\sin(\sqrt{24}\pi t)\cos(12\pi x_1 )\cos(2\pi x_2) \cos(2\pi x_3), \qquad \text{in} \quad \Omega_2.
\end{cases}
\end{align}

\end{exam}

Since the solution is only piecewise smooth, the sparse grid DG method is not expected to have good convergence rate. Therefore, we only show the convergence result obtained by the adaptive method in Table \ref{tab:dis-coeff-adaptive} for both 2D and 3D. In addition, the adaptive result with the parameter $N=8$ and $\epsilon=1\times10^{-4}$ in 2D is shown in Fig. \ref{fig:dis-coeff-2D}. There are fewer DoFs in the $x_1$ direction since the solution is smooth in that direction, and as expected, there are more DoFs located in the subdomain $\Omega_1$ than that in $\Omega \backslash \Omega_1$. 
\begin{table}[!hbp]
\centering
\caption{Example \ref{exam:03-dis-coeff}. discontinuous coefficient, adaptive sparse grid DG, 2D and 3D. $t=0.01$.}
\label{tab:dis-coeff-adaptive}
\begin{tabular}{c|c|c|c|c|c}
  \hline
  & $\epsilon$ & DoF & $L^2$-error & $R_{\textrm{DoF}}$ & $R_{\epsilon}$ \\
  \hline
\multirow{4}{3em}{$d=2$}
& 1e-1 & 480 & 2.93e-4 & - & - \\
& 1e-2 & 1088 & 8.43e-5 & 1.52 & 0.54 \\
& 1e-3 & 2240 & 8.46e-6 & 3.18 & 1.00 \\
& 1e-4 & 4224 & 1.05e-6 & 3.29 & 0.91 \\
\hline
\multirow{4}{3em}{$d=3$}
& 1e-1 & 2304 & 5.59e-4 & - & - \\
& 1e-2 & 7040 & 1.28e-4 & 1.32 & 0.64 \\
& 1e-3 & 18176 & 1.65e-5 & 2.16 & 0.89 \\
& 1e-4 & 41472 & 1.55e-6 & 2.87 & 1.03 \\
\hline
\end{tabular}
\end{table}

\begin{figure}
    \centering
    \subfigure[numerical solution]{
    \begin{minipage}[b]{0.46\textwidth}
    \includegraphics[width=1\textwidth]{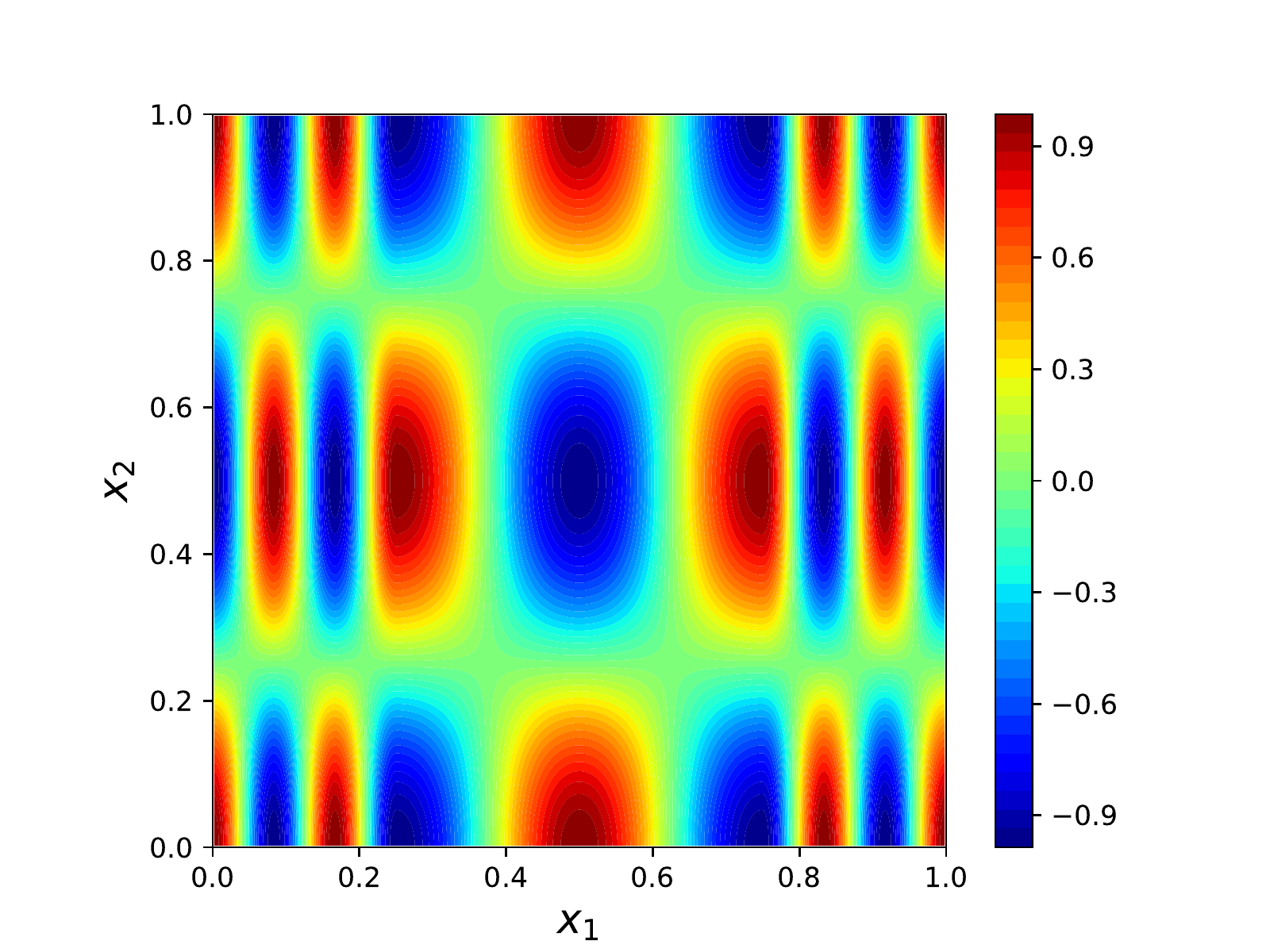}
    \end{minipage}
    }
    \subfigure[centers of active elements]{
    \begin{minipage}[b]{0.46\textwidth}    
    \includegraphics[width=1\textwidth]{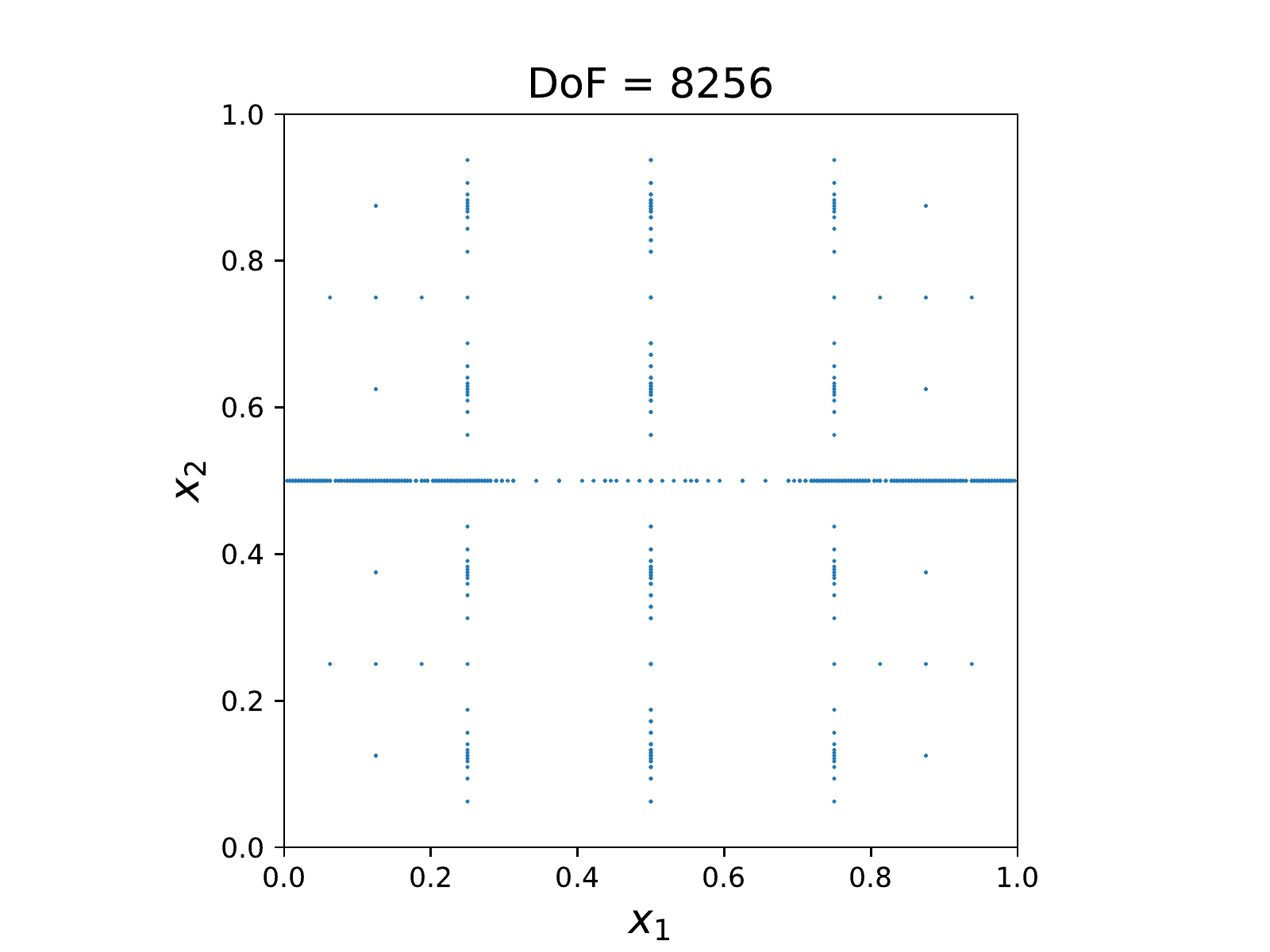}
    \end{minipage}
    }    
    \caption{Example \ref{exam:03-dis-coeff}: Discontinuous coefficient in 2D at $t=0.1$. Adaptive sparse grid DG with $N=8$ and $\epsilon=10^{-4}$.}
    \label{fig:dis-coeff-2D}
\end{figure}




\begin{exam}[Expanding wave in homogeneous medium]\label{exam:ring}
We consider the wave equation with constant wave speed $c=1$ on the computational domain $\Omega = [0, 1]^d$. The homogeneous Neumann boundary conditions are used on all   boundaries. The initial condition is taken as
\begin{equation}
	u(\bx, 0) = 0, \quad u_t(\bx, 0)  = 100e^{-500r^2}
\end{equation}
with $r=(\sum_{i=1}^d x_i^2)^{1/2}$ being the radius.

For small time $t$ (before the wave front touch the outside boundary), the exact solution in 2D can be represented by an integral which is derived by Hadamard's method of descent:
\begin{equation}
	u(x_1,x_2,t) = \frac{1}{2\pi}\iint_{\rho<t} \frac{100e^{-500(y_1^2+y_2^2)}}{\sqrt{t^2-\rho^2}}dy_1dy_2
\end{equation}
with $\rho:=\sqrt{(y_1-x_1)^2+(y_2-x_2)^2}$
and then computed by using numerical integrations with sufficiently small error tolerance. For $d=3$, there exists the analytic solution:
\begin{equation}
	u(\bx,t) = \frac{1}{20r} \brac{e^{-500 (t-r)^2} - e^{-500 (t+r)^2}}.
\end{equation}

The numerical results for 2D and 3D are presented in Fig. \ref{fig:ring-2D} and Fig. \ref{fig:ring-3D}. In both cases, our numerical solutions coincide with the exact solutions quite well. The $L^{\infty}$ errors between the numerical and the exact solutions at $t=0.5$ are $9.06\times10^{-5}$ and $6.79\times10^{-4}$ for 2D and 3D, which are both in  the same magnitude as the adaptive parameter $\epsilon=1\times10^{-4}$. This indicates that our adaptive algorithm controls the error really well. The DoFs are 14896 and 188672 for 2D and 3D. It can be also observed that the active elements in 3D are more ``sparse'' than 2D. This is a numerical evidence that the Huyghens principle only holds for wave equations in odd dimensions.

\end{exam}

\begin{figure}
    \centering
    \subfigure[numerical solution]{
    \begin{minipage}[b]{0.46\textwidth}
    \includegraphics[width=1\textwidth]{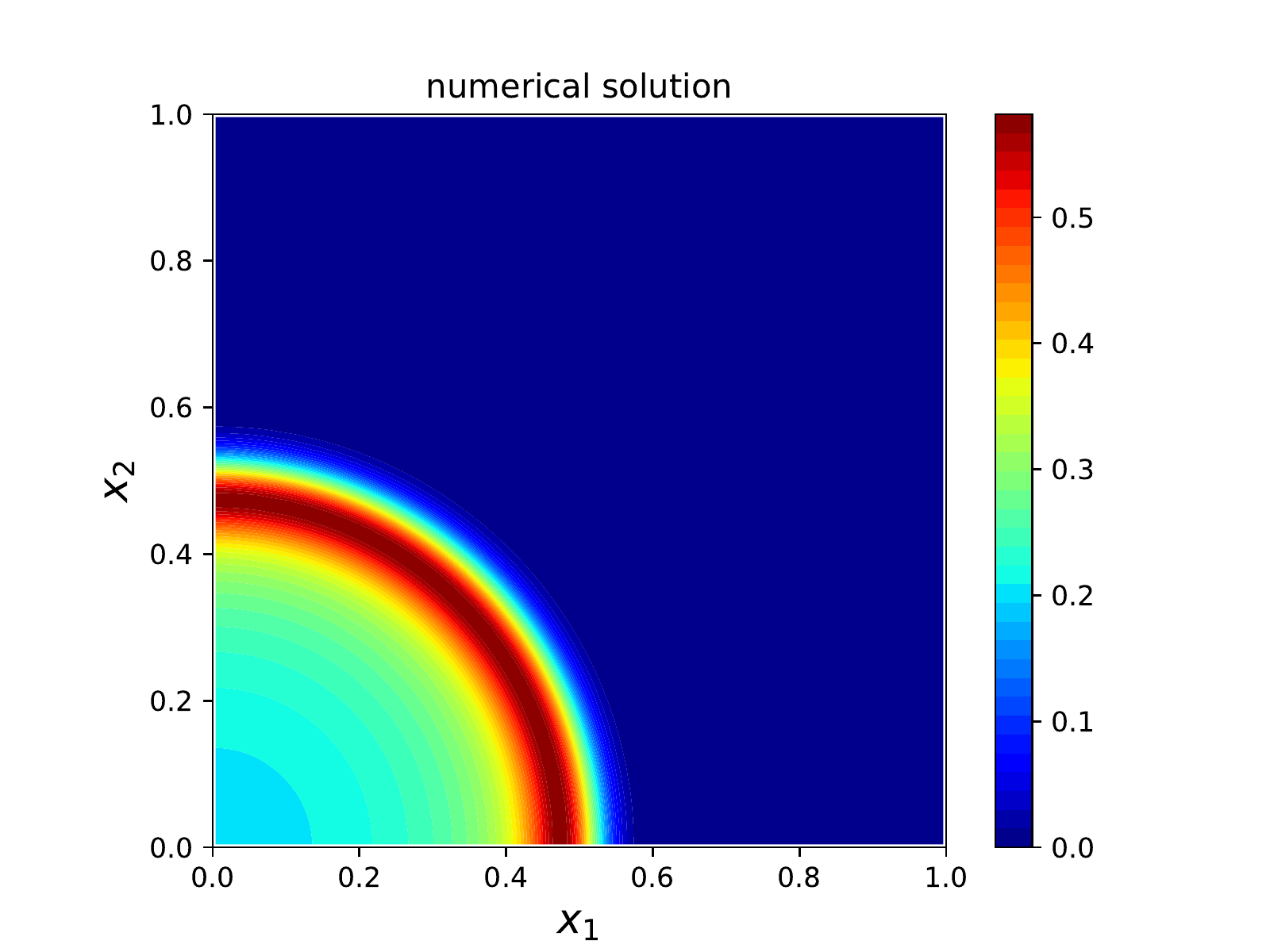}
    \end{minipage}
    }
    \subfigure[1D cut along diagonal $x_1=x_2$]{
    \begin{minipage}[b]{0.46\textwidth}    
    \includegraphics[width=1\textwidth]{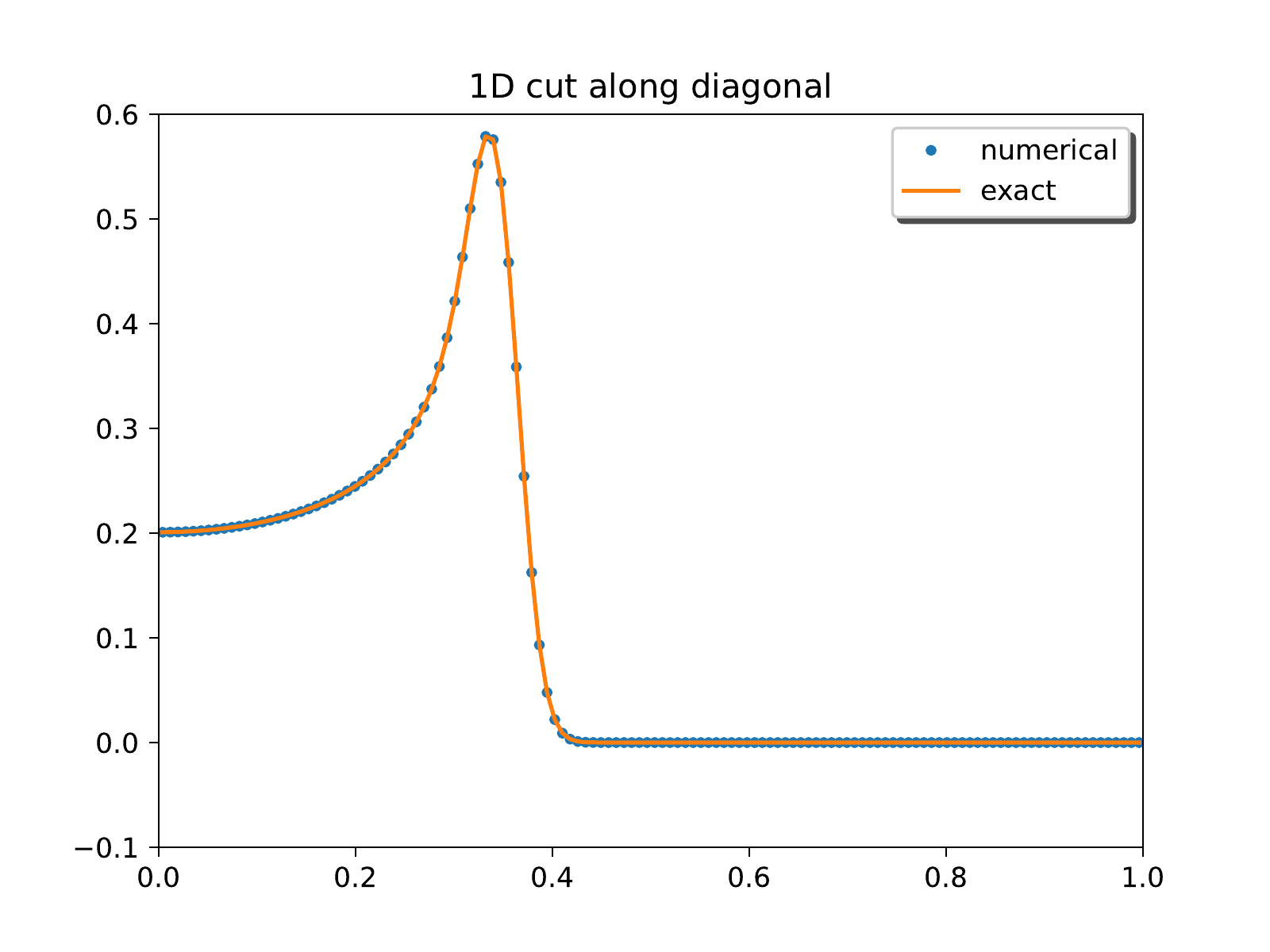}
    \end{minipage}
    }
    \bigskip
    \subfigure[error between exact and numerical solutions]{
    \begin{minipage}[b]{0.46\textwidth}
    \includegraphics[width=1\textwidth]{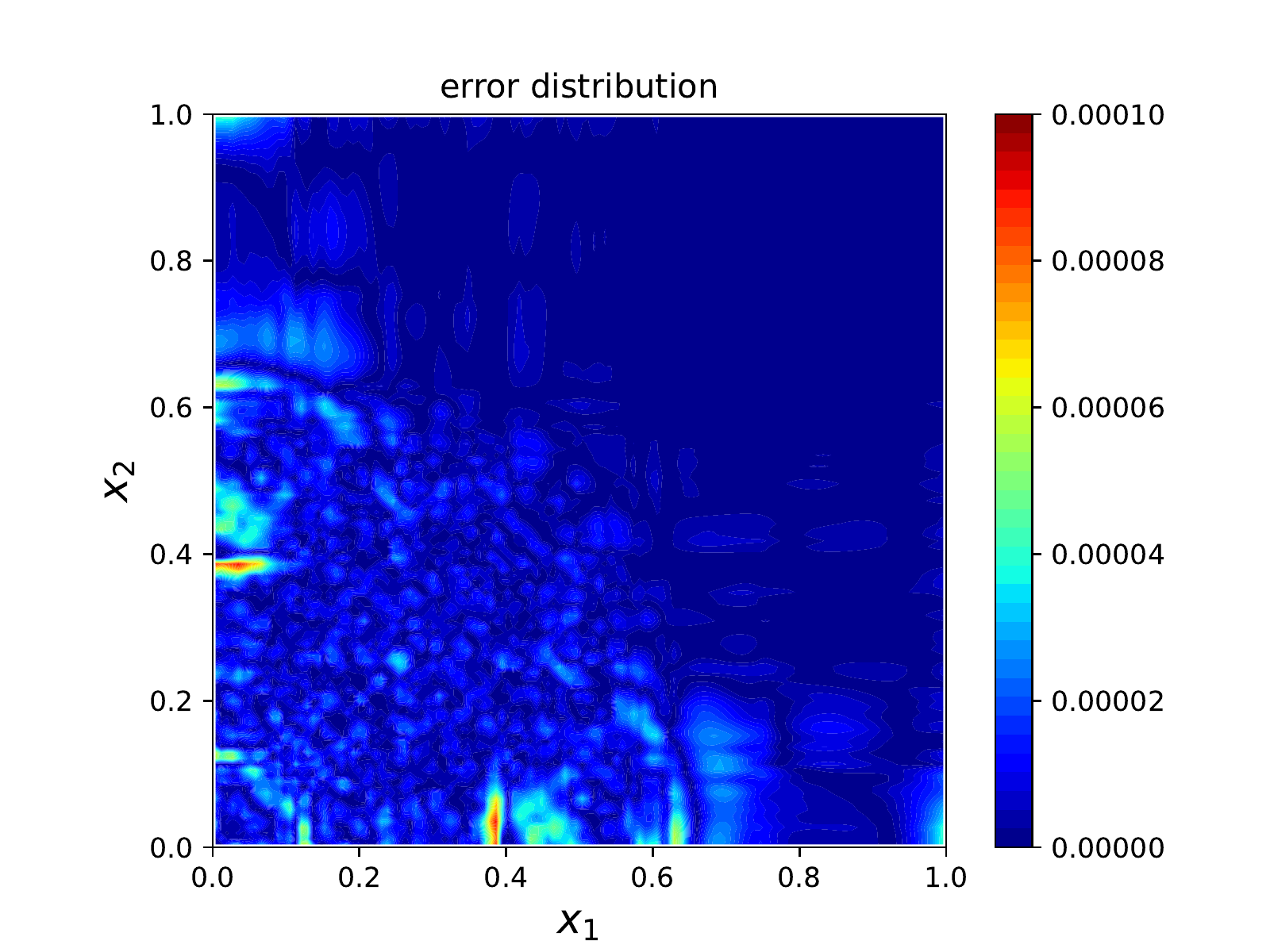}
    \end{minipage}
    }
    \subfigure[centers of active elements]{
    \begin{minipage}[b]{0.46\textwidth}    
    \includegraphics[width=1\textwidth]{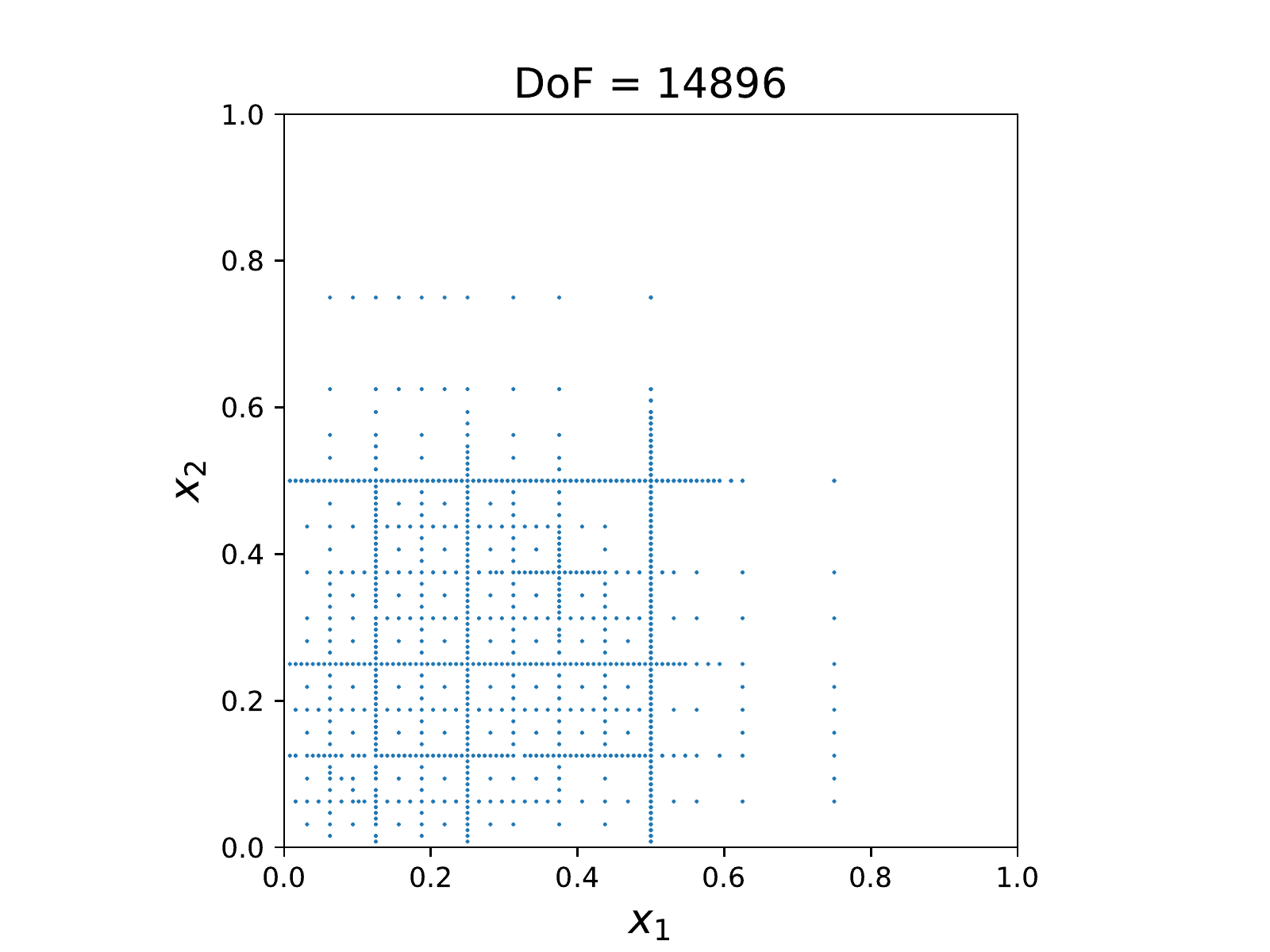}
    \end{minipage}
    }        
    \caption{Example \ref{exam:ring}: Expanding wave in homogeneous medium in 2D at $t=0.5$. Adaptive sparse grid DG. $N=7$ and $\epsilon=10^{-4}$.}
    \label{fig:ring-2D}
\end{figure}

\begin{figure}
    \centering
    \subfigure[numerical solution cut in 2D along $x_3=0$]{
    \begin{minipage}[b]{0.46\textwidth}
    \includegraphics[width=1\textwidth]{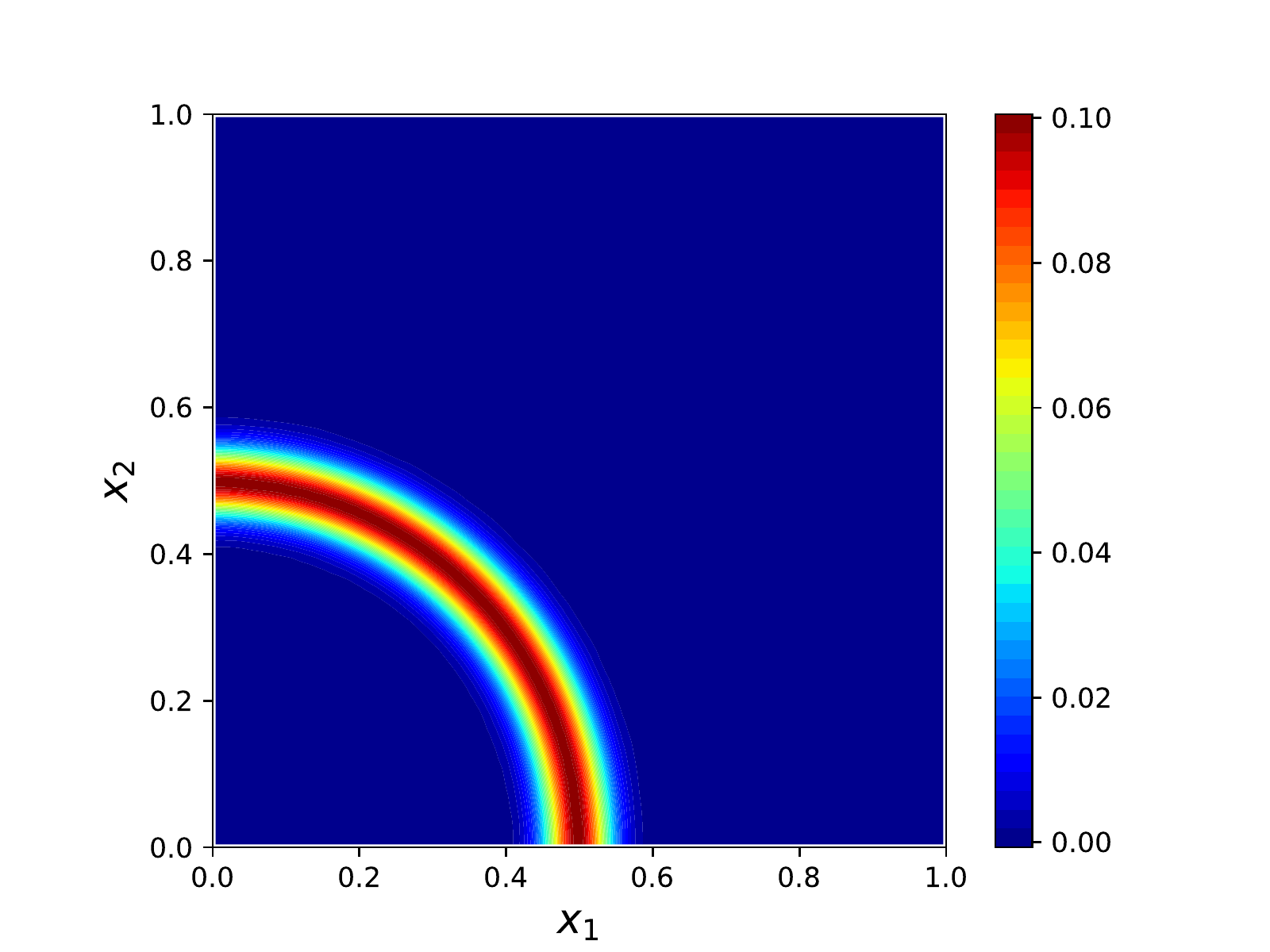}
    \end{minipage}
    }
    \subfigure[1D cut along $x_1=x_2$ and $x_3=0$]{
    \begin{minipage}[b]{0.46\textwidth}    
    \includegraphics[width=1\textwidth]{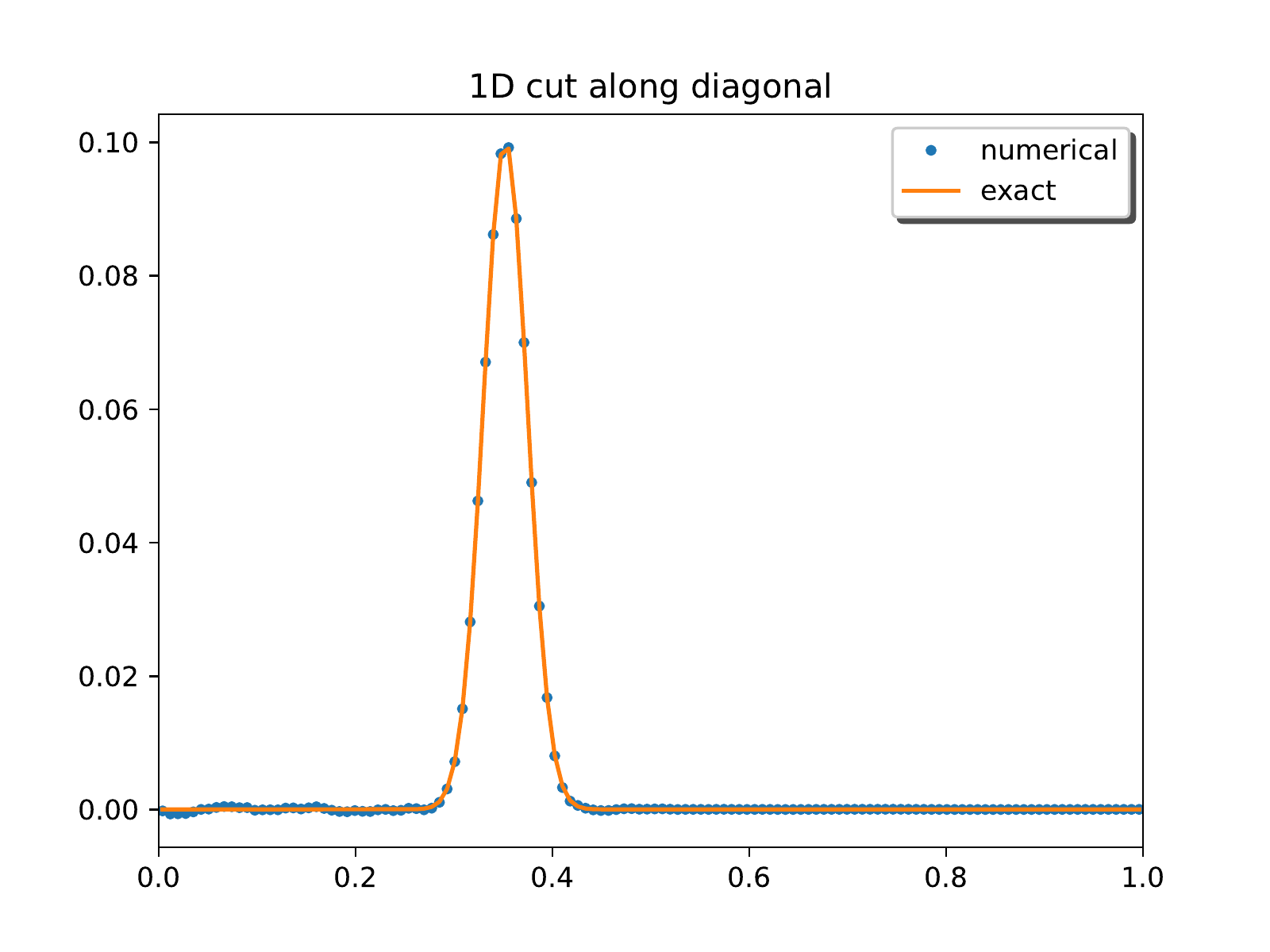}
    \end{minipage}
    }
    \bigskip
    \subfigure[centers of active elements in 3D]{
    \begin{minipage}[b]{0.46\textwidth}
    \includegraphics[width=1\textwidth]{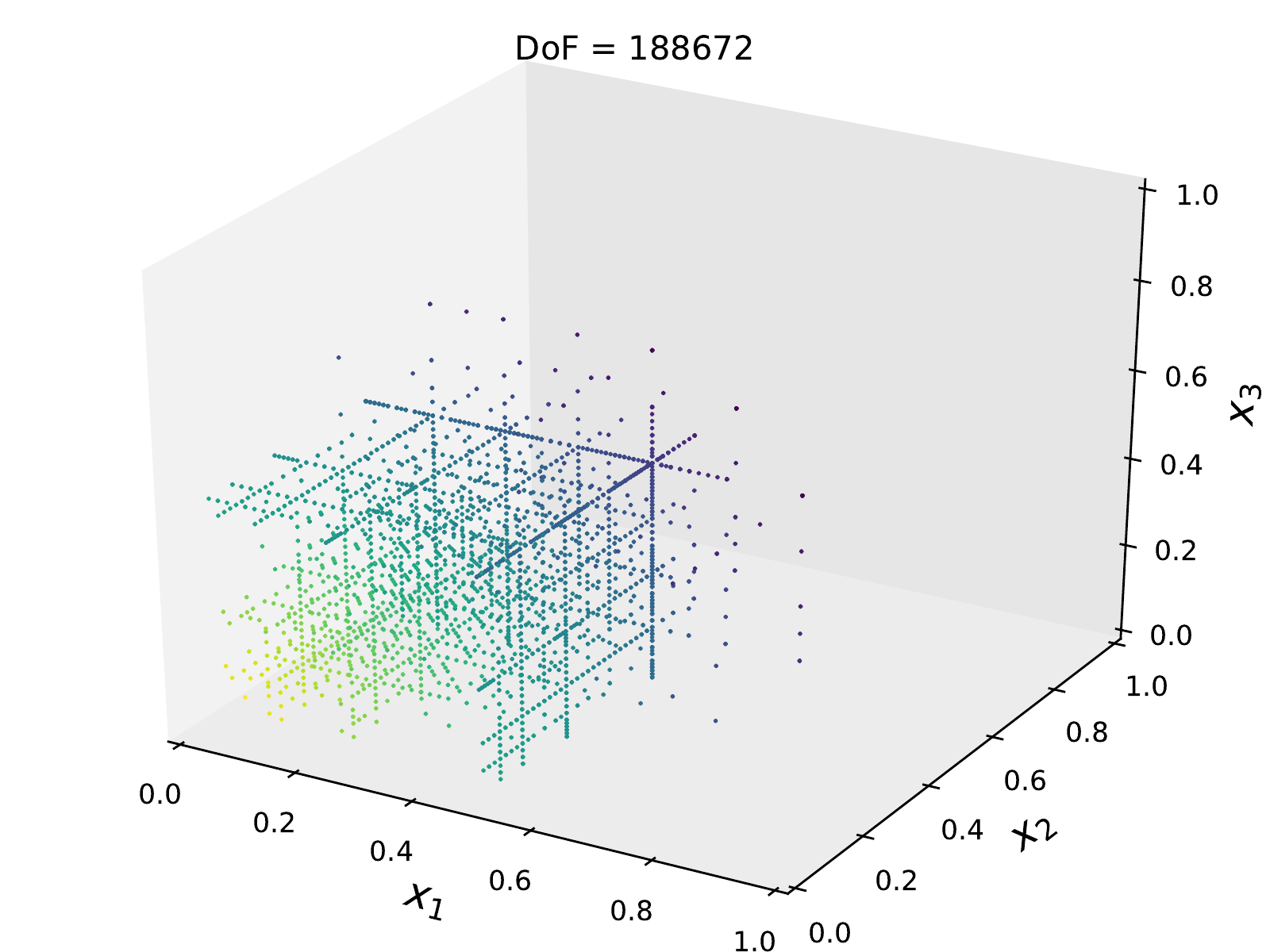}
    \end{minipage}
    }
    \subfigure[centers of active elements on $x_3=0.5$]{
    \begin{minipage}[b]{0.46\textwidth}    
    \includegraphics[width=1\textwidth]{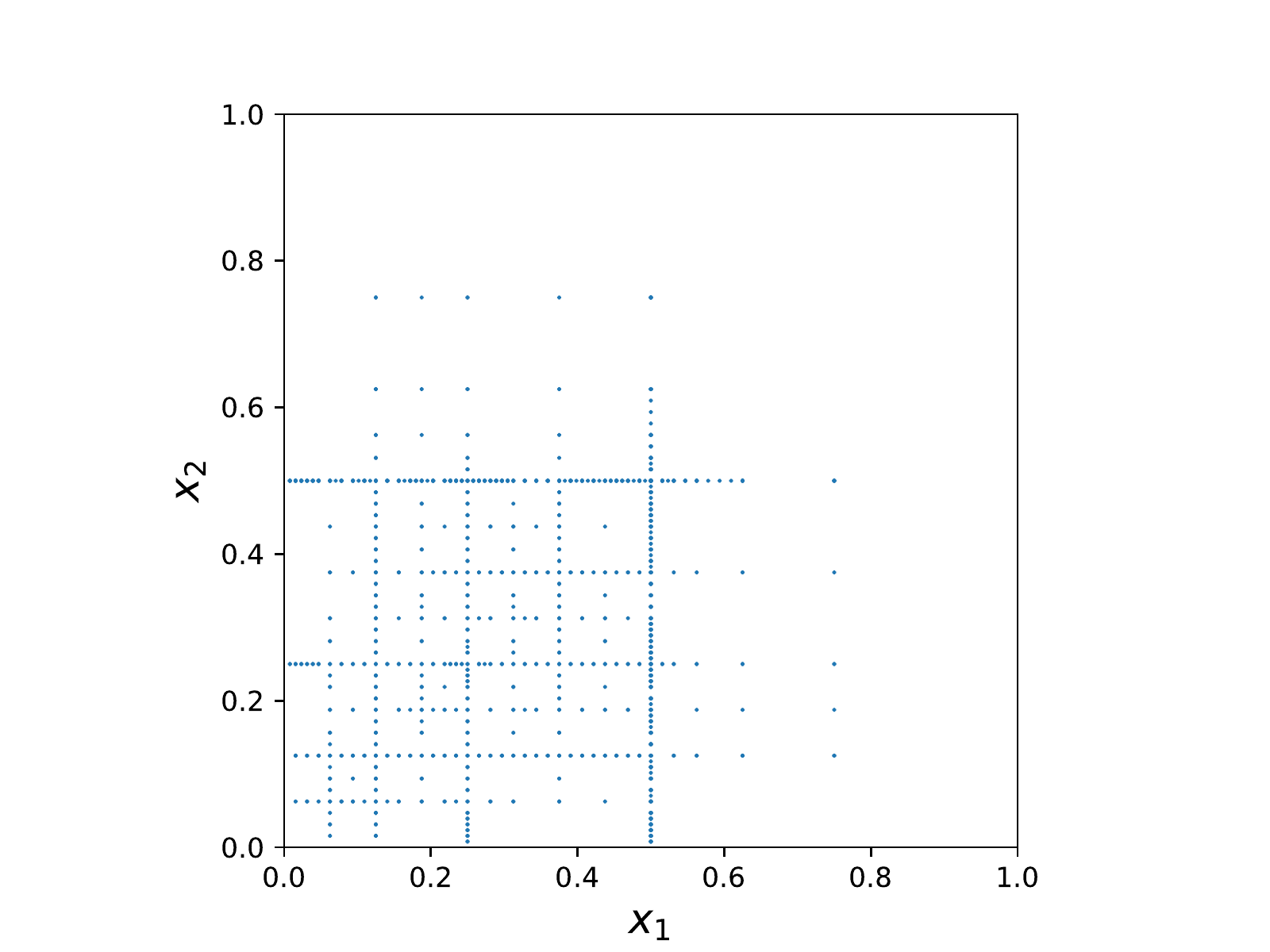}
    \end{minipage}
    }        
    \caption{Example \ref{exam:ring}: Expanding wave in homogeneous medium in 3D at $t=0.5$. Adaptive sparse grid DG.  $N=7$ and $\epsilon=10^{-4}$.}
    \label{fig:ring-3D}
\end{figure}

\begin{exam}[Isotropic wave propagation in heterogeneous media]\label{exam:heter-media}

We consider the wave equation with discontinuous coefficient on the computational domain $\Omega = [0, 1]^d$ for $d=2$ and 3 \cite{chou2014wave}.
\begin{align}
c^2=
\begin{cases}
\frac{1}{4}, \qquad \text{if} \quad 0.35 \le {x}_1 \le 0.65, \\
1, \qquad \text{otherwise}.
\end{cases}
\end{align}
Note that the jump in material coefficient is not aligned with the cell interface on $\Omega_N.$

For both 2D and 3D case, the initial conditions are taken as
\begin{align}\label{example5_2}
u(\mathbf{x}, 0)  = 0, \quad
u_t(\mathbf{x}, 0)   = 100e^{-500r^2}.
\end{align}
with $r=\brac{\sum_{i=1}^d(x_i-\frac12)^2}^{\frac12}$.
The zero Dirichlet boundary conditions are used.
\end{exam}

The profiles and centers of active elements obtained by the adaptive scheme are shown in Fig. \ref{fig:heter-media-2D} for 2D and Fig. \ref{fig:heter-media-3D} for 3D. We see that the wave fronts propagate at different speeds in these two media and our adaptive scheme capture this phenomenon and obtain comparable results to the literature \cite{chou2014wave}.
\begin{figure}
    \centering
    \subfigure[solution profile at $t=0.1$ ]{
    \begin{minipage}[b]{0.46\textwidth}
    \includegraphics[width=1\textwidth]{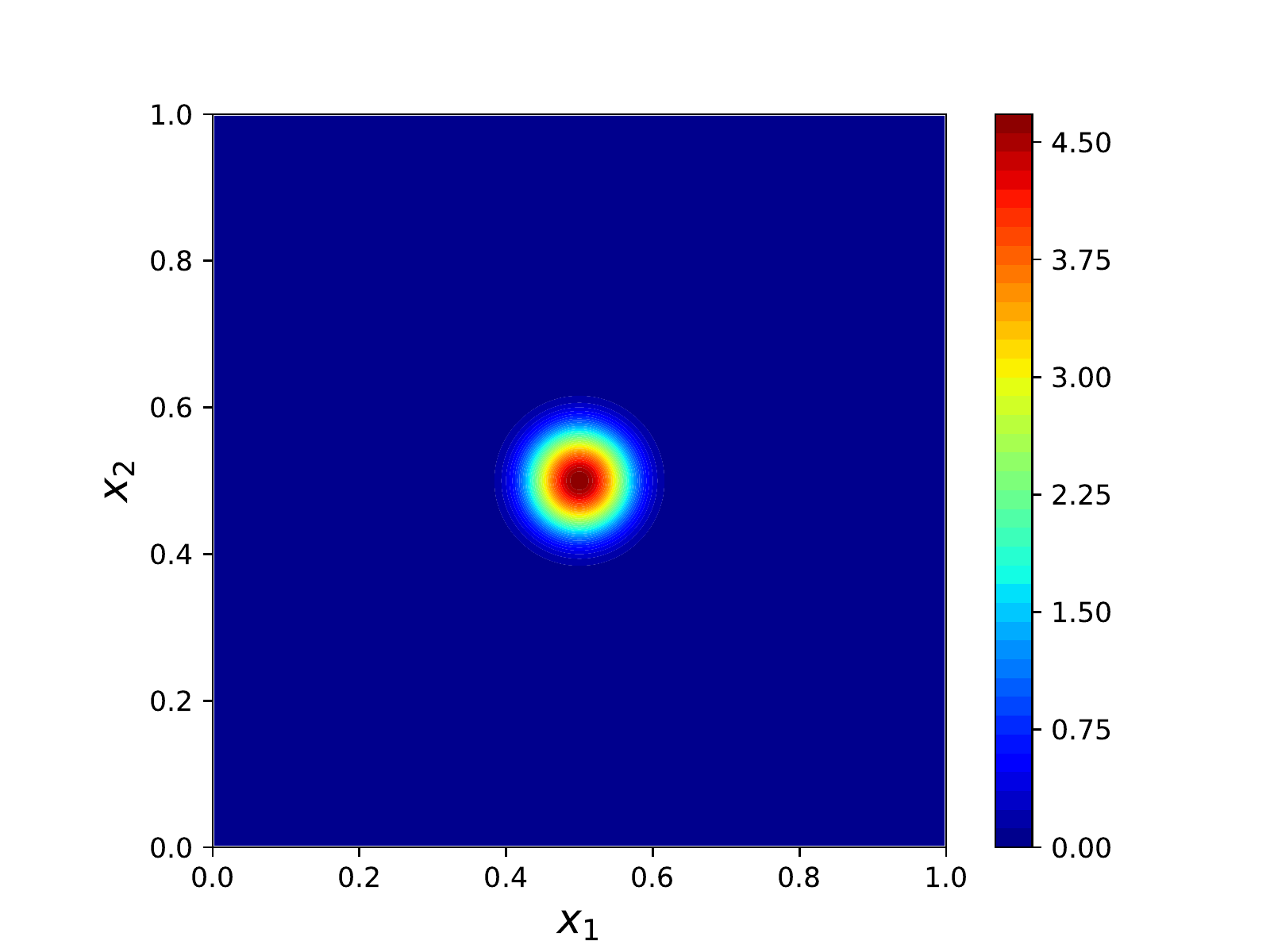}
    \end{minipage}
    }
    \subfigure[centers of active elements at $t=0.1$]{
    \begin{minipage}[b]{0.46\textwidth}    
    \includegraphics[width=1\textwidth]{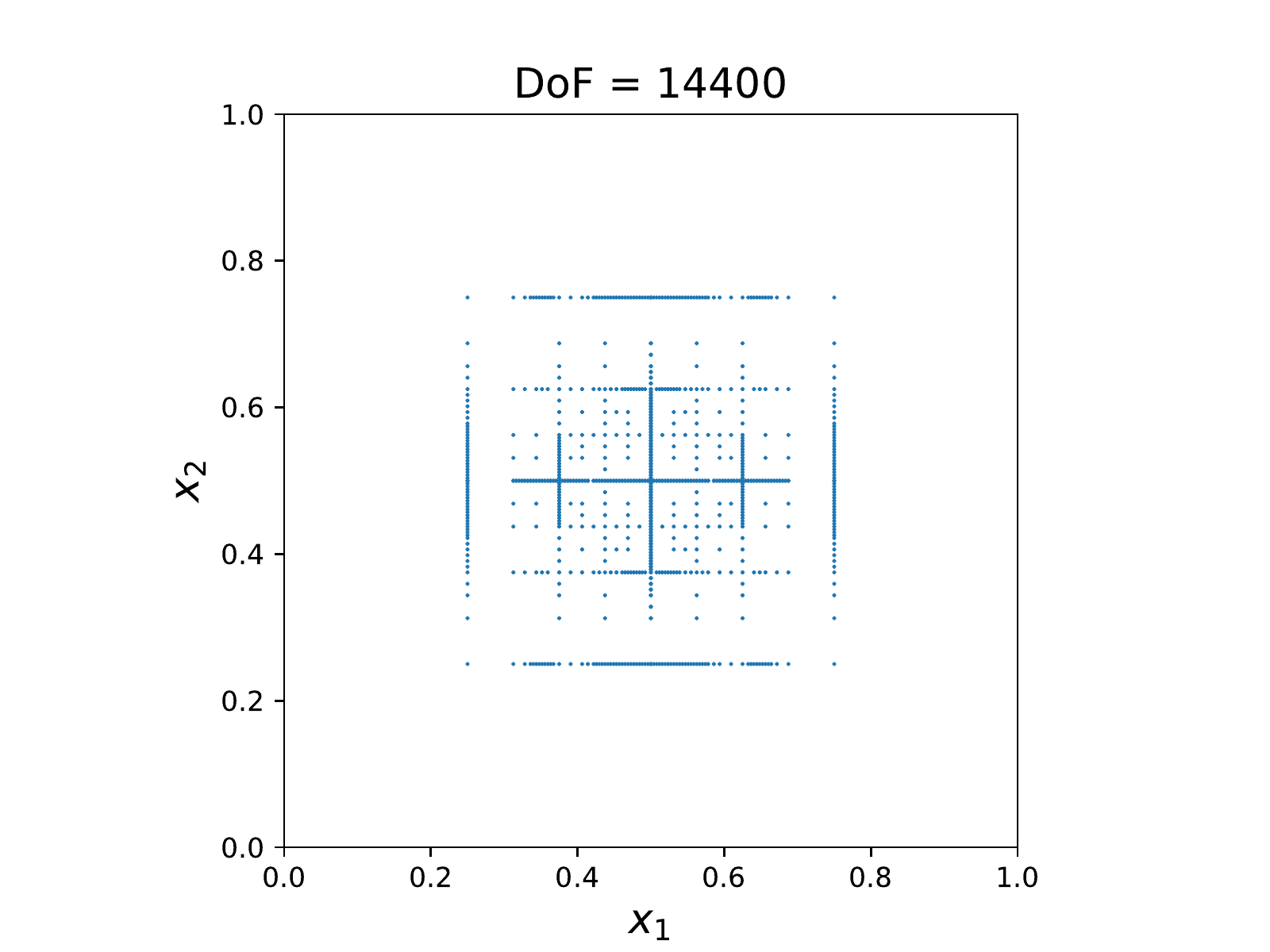}
    \end{minipage}
    }
    \bigskip
    \subfigure[solution profile at $t=0.3$]{
    \begin{minipage}[b]{0.46\textwidth}
    \includegraphics[width=1\textwidth]{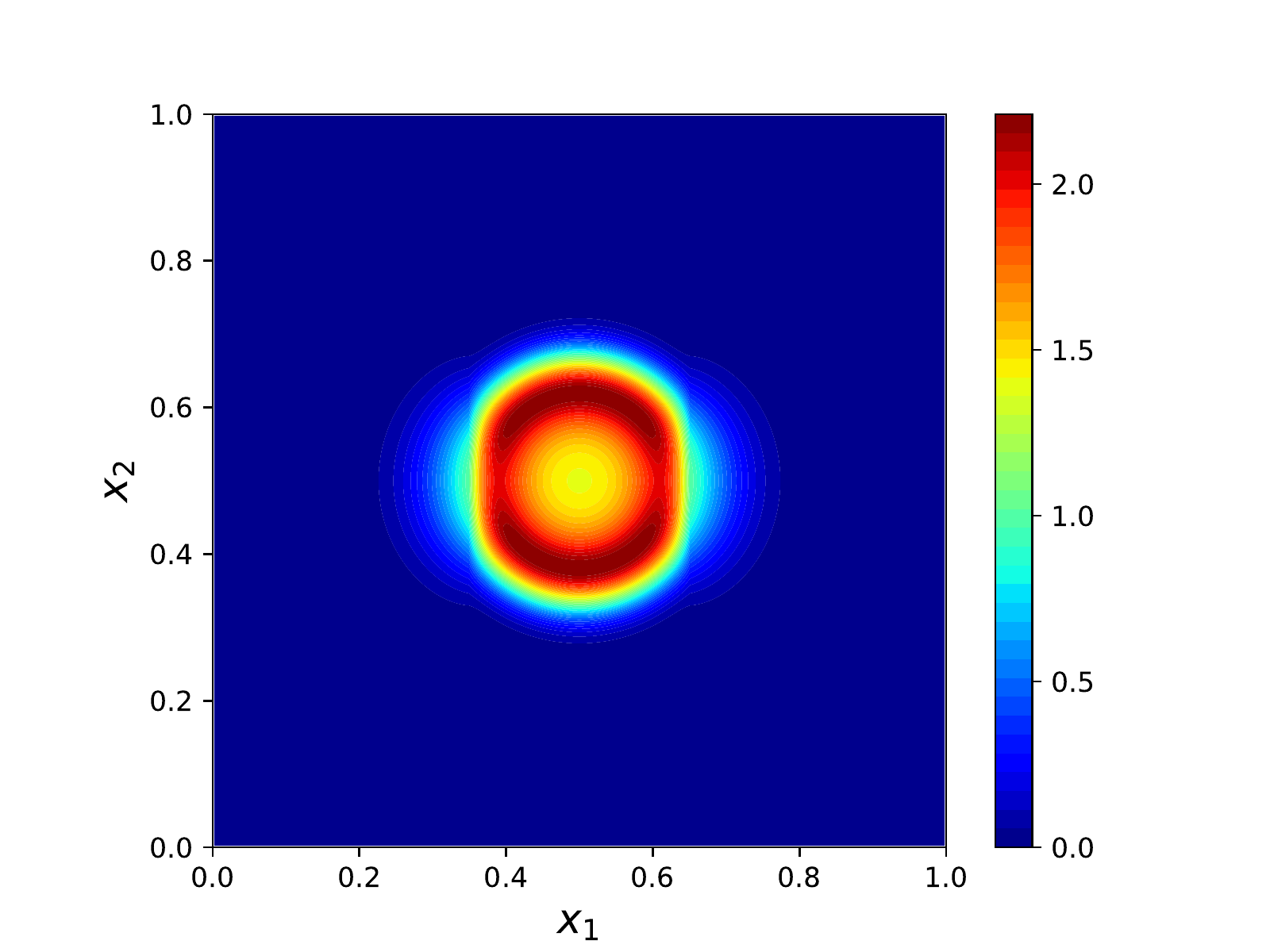}
    \end{minipage}
    }
    \subfigure[centers of active elements at $t=0.3$]{
    \begin{minipage}[b]{0.46\textwidth}    
    \includegraphics[width=1\textwidth]{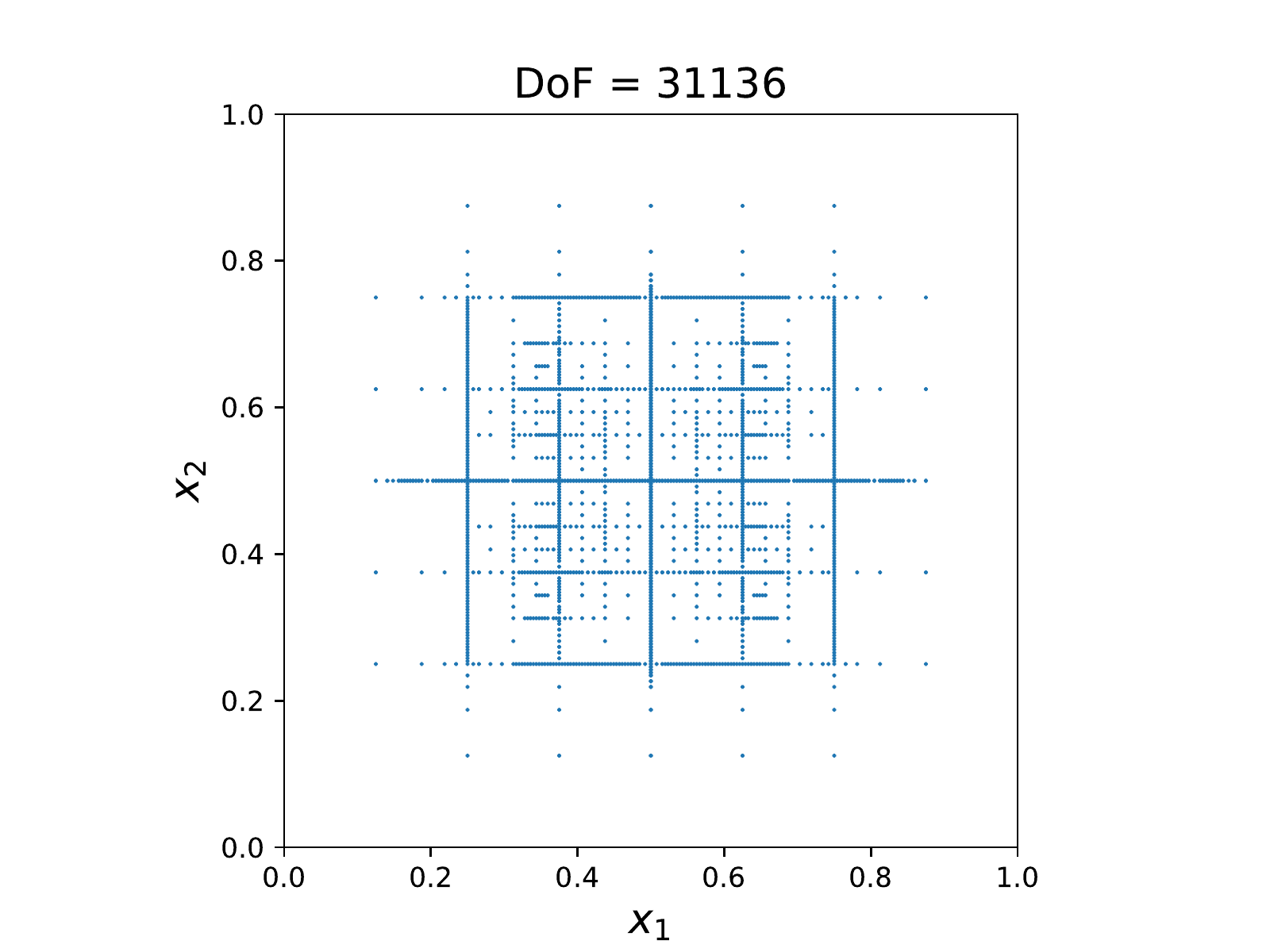}
    \end{minipage}
    }        
    \caption{Example \ref{exam:heter-media}: Isotropic wave propagation within heterogeneous media in 2D at $t=0.1$ and $t=0.3$. Adaptive sparse grid DG.  $N=8$ and $\epsilon=10^{-4}$. Left: solution profile; right: centers of active elements.}
    \label{fig:heter-media-2D}
\end{figure}

\begin{figure}
    \centering
    \subfigure[solution profile at $t=0.1$ (cut in 2D on $z=0.5$)]{
    \begin{minipage}[b]{0.46\textwidth}
    \includegraphics[width=1\textwidth]{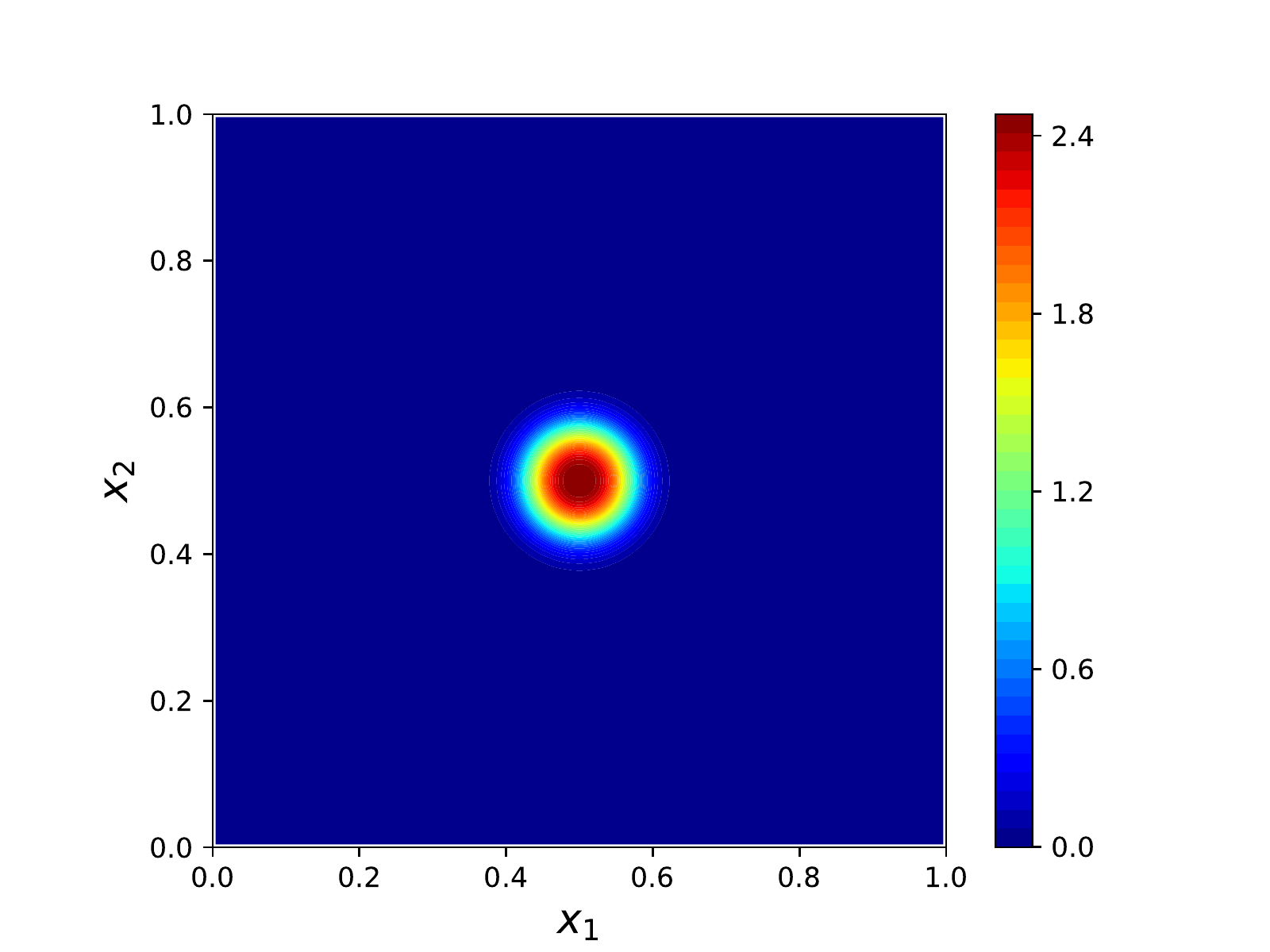}
    \end{minipage}
    }
    \subfigure[centers of active elements at $t=0.1$]{
    \begin{minipage}[b]{0.46\textwidth}    
    \includegraphics[width=1\textwidth]{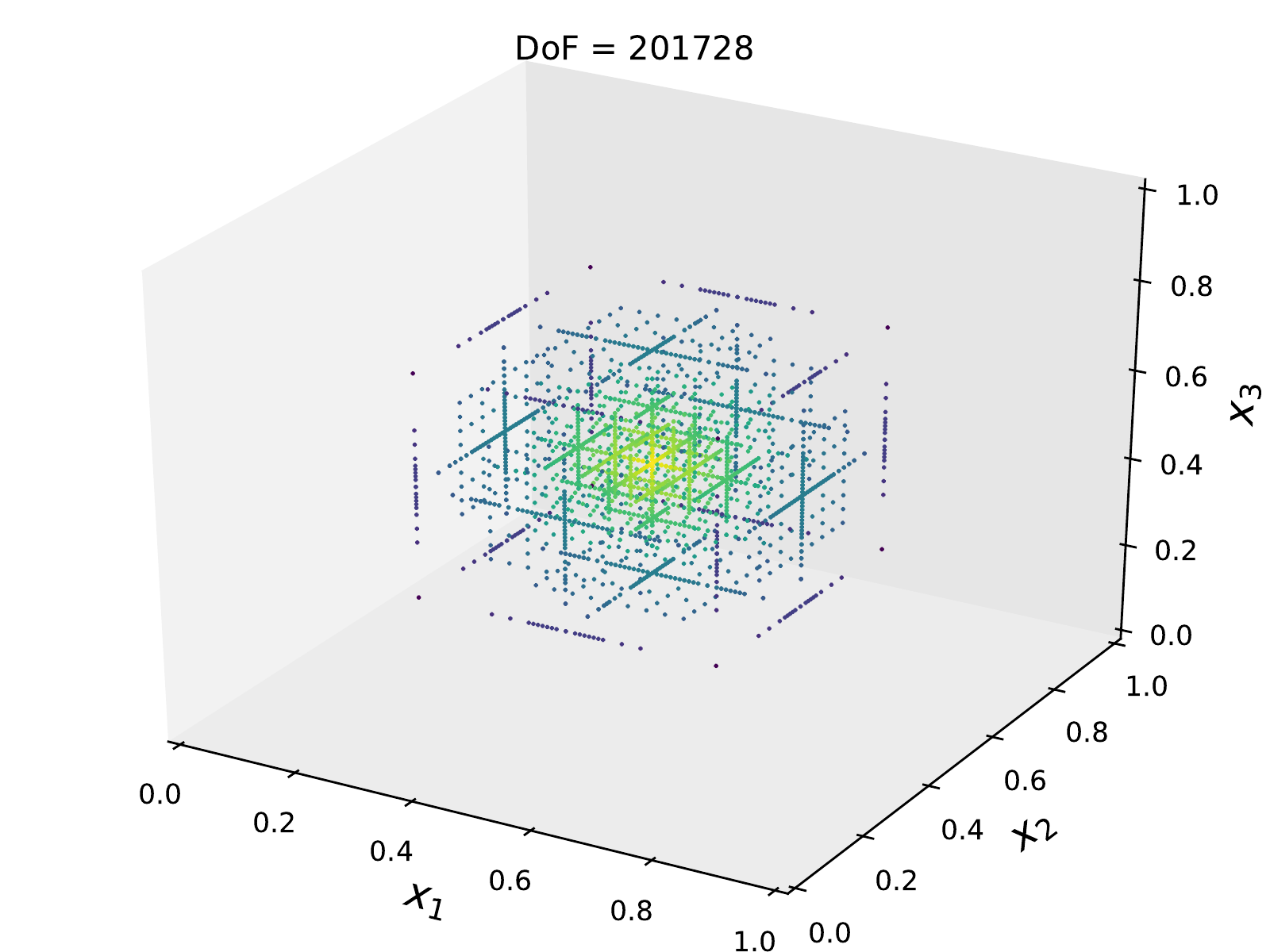}
    \end{minipage}
    }
    \bigskip
    \subfigure[solution profile at $t=0.3$ (cut in 2D on $z=0.5$)]{
    \begin{minipage}[b]{0.46\textwidth}
    \includegraphics[width=1\textwidth]{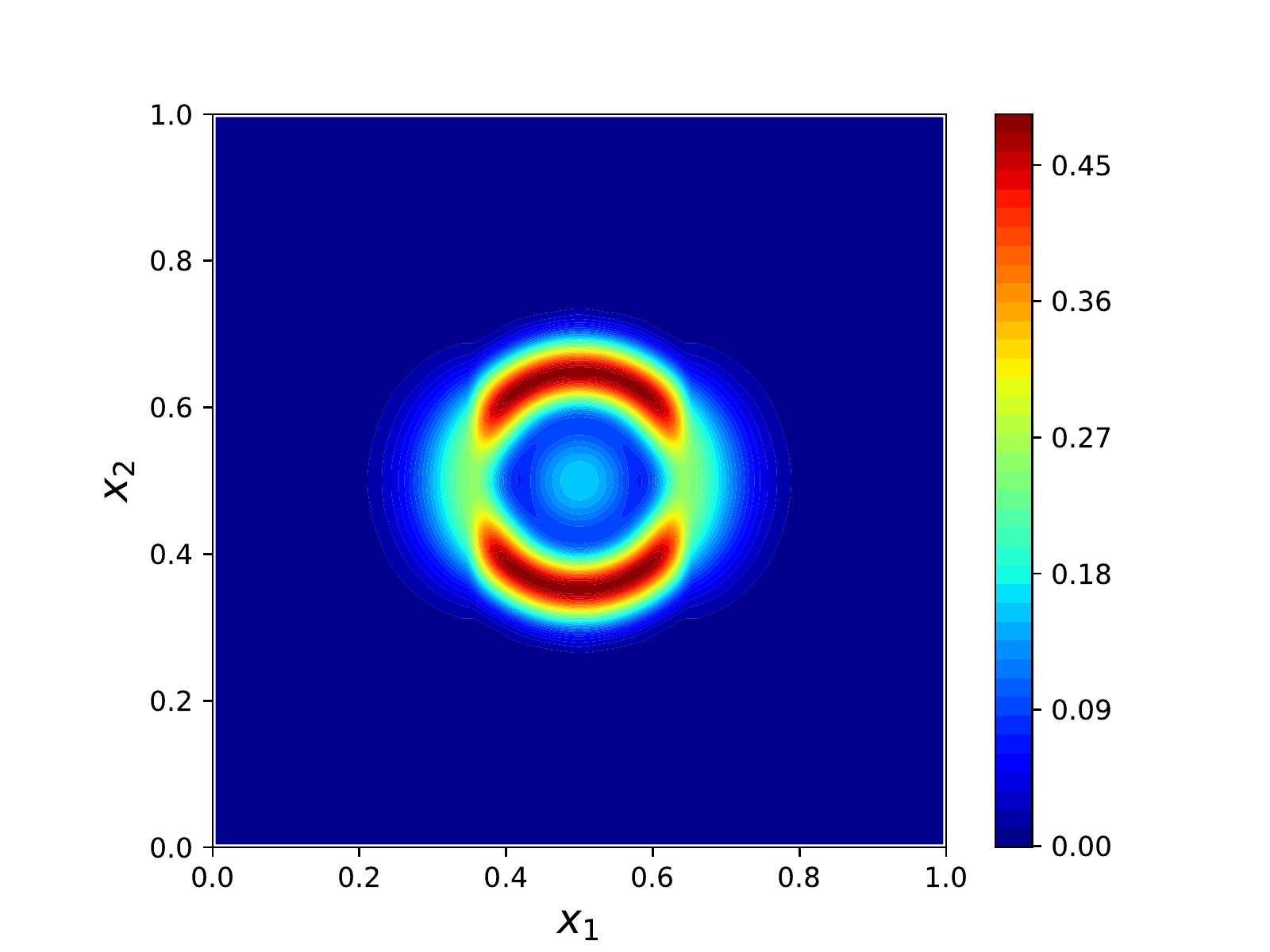}
    \end{minipage}
    }
    \subfigure[centers of active elements at $t=0.3$]{
    \begin{minipage}[b]{0.46\textwidth}    
    \includegraphics[width=1\textwidth]{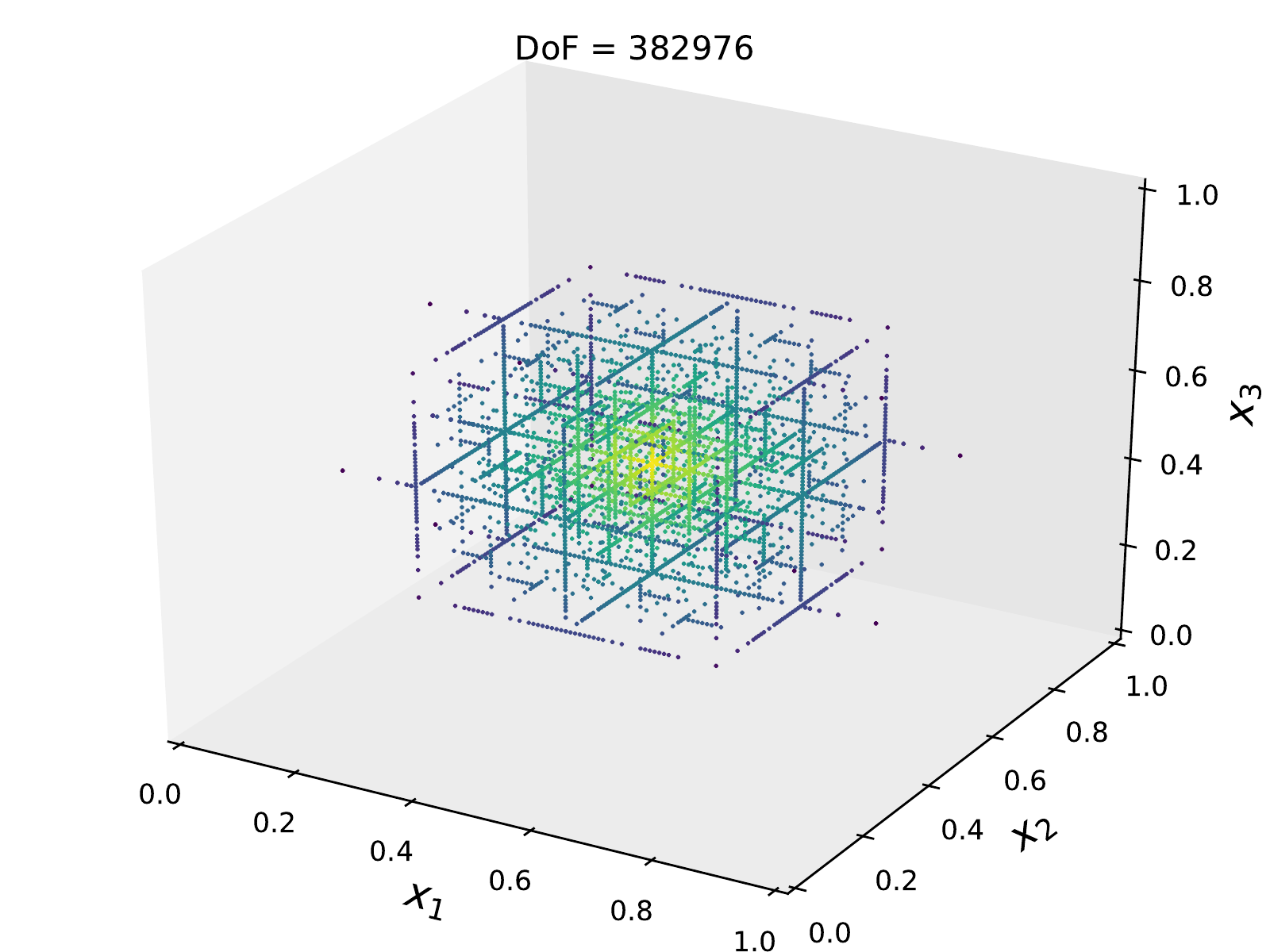}
    \end{minipage}
    }        
    \caption{Example \ref{exam:heter-media}: Isotropic wave propagation within heterogeneous media in 3D at $t=0.1$ and $t=0.3$. Adaptive sparse grid DG.  $N=7$ and $\epsilon=10^{-4}$. Left: solution profile; right: centers of active elements.}
    \label{fig:heter-media-3D}
\end{figure}

\section{Conclusion}
\label{sec:conclusion}

In this paper, we develop an adaptive multiresolution DG scheme for wave equations in second order form in multi-dimensions.  Our method can achieve similar computational complexity as the sparse grid DG method for smooth solutions like those proposed for equations in first order form \cite{guo2016transport, guo2017adaptive}. Extensive numerical tests in 2D and 3D verify the accuracy and robustness of the adaptive schemes for smooth and piecewise smooth wave propagation speed. Though the formulation is based on IPDG scheme for scalar wave equation, it can be extended to other DG method and other wave applications. Future work includes extensions to other boundary conditions and investigation on stability of schemes with interpolation. 
In an effort for promoting reproducible research, the code generating the results in this paper  can be found at the github link: \url{https://github.com/JuntaoHuang/adaptive-multiresolution-DG}.

\section*{Acknowledgment}
We would like to thank Daniel Appel\"{o} for discussions on numerical examples of wave propagation, Qi Tang and Kai Huang for the assistance and discussion in code implementation.

\bibliographystyle{abbrv}


\appendix

\section{Interpolation basis functions}

For completeness of the paper, we present details of the multiresolution interpolation basis functions, which are first introduced in \cite{tao2019collocation}. We will first focus on the case in which the interpolation points are imposed in the inner domain, as implemented in Table \ref{tab:smooth-sparse-2D-inner-pt}. Then we discuss the case in which the points includes the cell interface points. Here, we only discuss the case when $M=4$ and $M=5$. For $M=1,2,3$, we refer readers to the appendix in \cite{huang2019adaptive}.

The basis functions in $\tilde{W}_1$ are piecewise polynomials on $I_l:=(0,\half)$ and $I_r:=(\half,1)$. Note that the functions may be discontinuous at the interface $x=1/2$, thus $I_l$ and $I_r$ are  both defined to be open intervals. The basis functions in $\tilde{W}_1$ in this paper are all supported on one half interval $I_l$ or $I_r$ and vanish on the other half. For simplicity, we will only declare the function on its support. For example, $\psi_0(x)|_{I_r}$ gives the definition of $\psi_0$ on $I_r$ and indicates that $\psi_0$ vanishes on $I_l$.

\subsection{Interpolation points in the inner domain}

\subsubsection{$M=4$}

The interpolation points are
\begin{equation*}
    \tilde{X}_0 =  \{ \frac{1}{6},\frac{7}{24},\frac{1}{3},\frac{7}{12},\frac{2}{3} \}, 
    \quad \tilde{X}_1 =  \{ \frac{1}{12},\frac{7}{48},\frac{31}{48},\frac{19}{24},\frac{5}{6} \}.
\end{equation*}
The basis functions in $\tilde{W}_0^4$ and $\tilde{W}_1^4$ are
\begin{align*} 
\begin{array}{ll}
\phi_0(x) = \frac{4}{45} (3 x-2) (3 x-1) (12 x-7) (24 x-7), \\
\phi_1(x) = -\frac{512}{189}(3 x-2) (3 x-1) (6 x-1) (12 x-7), \\
\phi_2(x) = \frac{1}{3} (3 x-2) (6 x-1) (12 x-7) (24 x-7), \\
\phi_3(x) = -\frac{32}{105}(3 x-2) (3 x-1) (6 x-1) (24 x-7), \\
\phi_4(x) = \frac{1}{27} (3 x-1) (6 x-1) (12 x-7) (24 x-7).
\end{array} 
\end{align*}
and
\begin{align*} 
\begin{array}{l}
\psi_{0}(x)|_{I_l}= \frac{8}{45} (3 x-1) (6 x-1) (24 x-7) (48 x-7), \\
\psi_{1}(x)|_{I_l}= -\frac{1024}{189} (3 x-1) (6 x-1) (12 x-1) (24 x-7), \\ 
\psi_{2}(x)|_{I_r}= -\frac{1024}{189} (3 x-2) (6 x-5) (12 x-7) (24 x-19), \\
\psi_{3}(x)|_{I_r}= -\frac{64}{105} (3 x-2) (6 x-5) (12 x-7) (48 x-31), \\
\psi_{4}(x)|_{I_r}= \frac{2}{27} (3 x-2) (12 x-7) (24 x-19) (48 x-31)
\end{array}
\end{align*}

\subsubsection{$M=5$}

The interpolation points are
\begin{equation*}
    \tilde{X}_0 =  \{ \frac{1}{12},\frac{1}{6},\frac{7}{24},\frac{1}{3},\frac{7}{12},\frac{2}{3} \}, 
    \quad \tilde{X}_1 =  \{ \frac{7}{48},\frac{1}{24},\frac{31}{48},\frac{19}{24},\frac{5}{6},\frac{13}{24} \}.
\end{equation*}
The basis functions in $\tilde{W}_0^5$ and $\tilde{W}_1^5$ are
\begin{align*} 
\begin{array}{ll}
\phi_0(x) = \frac{1}{315} (-16) (3 x-2) (3 x-1) (6 x-1) (12 x-7) (24 x-7), \\
\phi_1(x) = \frac{4}{45} (3 x-2) (3 x-1) (12 x-7) (12 x-1) (24 x-7), \\
\phi_2(x) = -\frac{1024}{945}(3 x-2) (3 x-1) (6 x-1) (12 x-7) (12 x-1), \\
\phi_3(x) = \frac{1}{9} (3 x-2) (6 x-1) (12 x-7) (12 x-1) (24 x-7), \\
\phi_4(x) = -\frac{16}{315} (3 x-2) (3 x-1) (6 x-1) (12 x-1) (24 x-7), \\
\phi_5(x) = \frac{1}{189} (3 x-1) (6 x-1) (12 x-7) (12 x-1) (24 x-7),
\end{array} 
\end{align*}
and
\begin{align*} 
\begin{array}{l}
\psi_{0}(x)|_{I_l}= -\frac{2048}{945} (3 x-1) (6 x-1) (12 x-1) (24 x-7) (24 x-1), \\
\psi_{1}(x)|_{I_l}= -\frac{32}{315} (3 x-1) (6 x-1) (12 x-1) (24 x-7) (48 x-7), \\ 
\psi_{2}(x)|_{I_r}= -\frac{2048}{945} (3 x-2) (6 x-5) (12 x-7) (24 x-19) (24 x-13), \\
\psi_{3}(x)|_{I_r}= -\frac{32}{315} (3 x-2) (6 x-5) (12 x-7) (24 x-13) (48 x-31), \\
\psi_{4}(x)|_{I_r}= \frac{2}{189} (3 x-2) (12 x-7) (24 x-19) (24 x-13) (48 x-31), \\
\psi_{5}(x)|_{I_r}= -\frac{32}{315} (3 x-2) (6 x-5) (12 x-7) (24 x-19) (48 x-31)
\end{array}
\end{align*}

\subsection{Interpolation points with the interface points}

\subsubsection{$M=4$}

The interpolation points are
\begin{equation*}
    \tilde{X}_0 =  \{ 0^+,\brac{\frac{1}{4}}^-, \brac{\frac{1}{2}}^-, \brac{\frac{3}{4}}^-, 1^- \}, 
    \quad \tilde{X}_1 =  \{ \brac{\frac{1}{8}}^-, \brac{\frac{3}{8}}^-, \brac{\frac{1}{2}}^+, \brac{\frac{5}{8}}^-, \brac{\frac{7}{8}}^- \}.
\end{equation*}
The basis functions in $\tilde{W}_0^4$ and $\tilde{W}_1^4$ are
\begin{align*} 
\begin{array}{ll}
\phi_0(x) = \frac{1}{3} (x-1) (2 x-1) (4 x-3) (4 x-1), \\
\phi_1(x) = -\frac{16}{3} (x-1) x (2 x-1) (4 x-3), \\
\phi_2(x) = 4 (x-1) x (4 x-3) (4 x-1), \\
\phi_3(x) = -\frac{16}{3} (x-1) x (2 x-1) (4 x-1), \\
\phi_4(x) = \frac{1}{3} x (2 x-1) (4 x-3) (4 x-1).
\end{array} 
\end{align*}
and
\begin{align*} 
\begin{array}{l}
\psi_{0}(x)|_{I_l}= -\frac{32}{3} x (2 x-1) (4 x-1) (8 x-3), \\
\psi_{1}(x)|_{I_l}= -\frac{32}{3} x (2 x-1) (4 x-1) (8 x-1), \\ 
\psi_{2}(x)|_{I_r}= \frac{2}{3} (x-1) (4 x-3) (8 x-7) (8 x-5), \\
\psi_{3}(x)|_{I_r}= -\frac{32}{3} (x-1) (2 x-1) (4 x-3) (8 x-7), \\
\psi_{4}(x)|_{I_r}= -\frac{32}{3} (-32) (x-1) (2 x-1) (4 x-3) (8 x-5)
\end{array}
\end{align*}

\subsubsection{$M=5$}

The interpolation points are
\begin{equation*}
    \tilde{X}_0 =  \{ 0^+,\frac{1}{5},\frac{2}{5},\frac{3}{5},\frac{4}{5}, 1^- \}, 
    \quad \tilde{X}_1 =  \{ \frac{1}{10},\frac{3}{10},\brac{\frac{1}{2}}^-,\brac{\frac{1}{2}}^+,\frac{7}{10},\frac{9}{10} \}.
\end{equation*}
The basis functions in $\tilde{W}_0^5$ and $\tilde{W}_1^5$ are
\begin{align*} 
\begin{array}{ll}
\phi_0(x) = -\frac{1}{24} (x-1) (5 x-4) (5 x-3) (5 x-2) (5 x-1), \\
\phi_1(x) = \frac{25}{24} (x-1) x (5 x-4) (5 x-3) (5 x-2), \\
\phi_2(x) = -\frac{25}{12} (x-1) x (5 x-4) (5 x-3) (5 x-1), \\
\phi_3(x) = \frac{25}{12} (x-1) x (5 x-4) (5 x-2) (5 x-1), \\
\phi_4(x) = -\frac{25}{24} (x-1) x (5 x-3) (5 x-2) (5 x-1), \\
\phi_5(x) = \frac{1}{24} x (5 x-4) (5 x-3) (5 x-2) (5 x-1),
\end{array} 
\end{align*}
and
\begin{align*} 
\begin{array}{l}
\psi_{0}(x)|_{I_l}= \frac{25}{3} x (2 x-1) (5 x-2) (5 x-1) (10 x-3), \\
\psi_{1}(x)|_{I_l}= \frac{50}{3} x (2 x-1) (5 x-2) (5 x-1) (10 x-1), \\ 
\psi_{2}(x)|_{I_r}= \frac{1}{3} x (5 x-2) (5 x-1) (10 x-3) (10 x-1), \\
\psi_{3}(x)|_{I_r}= -\frac{1}{3} (x-1) (5 x-4) (5 x-3) (10 x-9) (10 x-7), \\
\psi_{4}(x)|_{I_r}= -\frac{50}{3} (x-1) (2 x-1) (5 x-4) (5 x-3) (10 x-9), \\
\psi_{5}(x)|_{I_r}= -\frac{25}{3} (x-1) (2 x-1) (5 x-4) (5 x-3) (10 x-7).
\end{array}
\end{align*}


\end{document}